\documentclass{amsart}
\usepackage{amsmath,amssymb,latexsym,amsthm,graphicx}
\usepackage{amsmath}
\usepackage{color}
\usepackage{amsfonts}
\usepackage{amssymb}
\usepackage{graphicx}

\newtheorem{theorem}{Theorem}[section]
\newtheorem{problem}[theorem]{Problem}

\newtheorem{condition}[theorem]{Condition}
\newtheorem{lemma}[theorem]{Lemma}
\newtheorem{prop}[theorem]{Proposition}
\newtheorem{remark}[theorem]{Remark}

\newtheorem{corollary}[theorem]{Corollary}
\newtheorem{defi}[theorem]{Definition}

\begin{document}
\title{A dissipative time reversal technique for photo-acoustic tomography
in a cavity}
\author{Linh V. Nguyen and Leonid Kunyansky}

\begin{abstract}
We consider the inverse source problem arising in thermo- and photo-
acoustic tomography. It consists in reconstructing the initial pressure from
the boundary measurements of the acoustic wave. Our goal is to extend
versatile time reversal techniques to the case of perfectly reflecting
boundary of the domain. Standard time reversal works only if the solution of
the direct problem decays in time, which does not happen in the setup we
consider. We thus propose a novel time reversal technique with a
non-standard boundary condition. The error induced by this time reversal
technique satisfies the wave equation with a dissipative boundary condition
and, therefore, decays in time. For larger measurement times, this method
yields a close approximation; for smaller times, the first approximation can
be iteratively refined, resulting in a convergent Neumann series for the
approximation.
\end{abstract}

\maketitle


\section{Introduction\label{S:intro}}

We consider the inverse source problem arising in the thermo- and
photoacoustic tomography (TAT and PAT)
\cite{KrugerPAT,KrugerTAT,Oraev94,XW06}.
It consists in reconstructing the
initial pressure in the acoustic wave from the values of time-dependent
pressure measured on a surface, completely or partially surrounding the
object of interest \cite{KKun}. During the last decade, significant results
were obtained in solving this problem under the assumption that the wave
propagates in free space (see, for example \cite{FPR,Kun-expl,Kun-ser,Ng,
US,US-Num,AK,FHR,Nat12,XW05,Pala,PS02,Ha09} and reviews \cite{KKun,KKun1,Sch}
for additional references). Applicability of the free space approximation
depends on the type of the device(s) used to conduct the measurements: it is
valid if reflection of waves from the detectors can be neglected. There are,
however, a number of situations where this simple model is not applicable.
For example, when the object is surrounded by glass plates optically scanned
to measure the pressure, the waves experience multiple reflections as they
would in a resonant cavity \cite{Cox-Cavity,cox2008photo,Kunyansky-Cavity}.
As a result, the assumption indispensable in the analysis of the classical
TAT and PAT about the fast decrease of acoustic energy within the object,
can no longer be made.

Thus, a novel approach is needed to solve the inverse source problem posed
within a resonant cavity. The latter problem has attracted attention of
analysts only recently. A particular case of a rectangular resonant cavity
was considered in \cite{Cox-Cavity,cox2008photo,Kunyansky-Cavity}, and
several solutions that exploit symmetries of such a geometry were proposed.
In \cite{Holman-Kunyansky,Stefanov-Yang,Acosta}, more general acquisition
geometries were considered; one of the main questions investigated in these
papers is the applicability of various modifications of the time reversal
algorithm to the problem under consideration.

Time reversal was successfully used by multiple authors, both theoretically
and numerically, to solve the inverse source problem in the free space
setting (\cite{FPR,HKN,US,HTRE,US-Num,US-Brain,Homan}). In the latter case,
it consists in solving the wave equation backwards in time, within the
domain $\Omega$ surrounded by the detectors, using the measured data to pose
the Dirichlet boundary condition. In the simplest case of constant speed of
sound $c(x)\equiv c_{0},$ the time reversal is initialized with zero
conditions inside the domain (at time $t=T$ for sufficiently large $T)$.
Such approach works due to the fact that the solution of the direct problem
in $\mathbb{R}^{3}$ with $c(x)\equiv c_{0}$ vanishes within the domain in a
finite time $T_{0}(\Omega)$, due to the Huygens principle (and so time $T$
is chosen to be greater or equal to $T_{0}(\Omega)).$ In 2D and/or when $c(x)
$ is not constant, more sophisticated methods are used to initialize time
reversal \cite{HKN,US,US-Num}; however, all these techniques require a time
decay of the solution of the direct problem.

The inverse source problem in a cavity with \emph{partially} reflecting
walls was solved in \cite{Acosta} using a time reversal approach. In this
case, due to the outflow of energy through the boundary of the domain,
solution of the direct problem still exhibits time decay, and, with some
modifications, time reversal can still be applied. A more difficult case is
that of \emph{perfectly} reflecting boundary (modeled by zero Neumann
condition on the boundary). In this case, the direct problems is energy
preserving, and acoustic oscillations in the cavity continue forever (here
we neglect the effects of attenuation of waves in the tissues). This
immediately makes classical time reversal inapplicable: the correct pressure
and its time derivative inside the domain are unknown at any time. The error
made when replacing these functions by any crude guess is not small. Under
the standard Dirichlet boundary condition, this error does not decrease as $t
$ approaches $0$.

Several attempts were made to modify the Dirichlet boundary condition so as
to control this error. In \cite{Holman-Kunyansky}, the boundary condition is
multiplied by a \ smooth cut off function $\eta(t/T),$ where $\eta(\tau)$
vanishes at $\tau=1$ with at least several derivatives. It was shown that
under certain conditions on the eigenvalues of the Dirichlet and Neumann
Laplacians on $\Omega,$ the approximate solutions corresponding to
measurement time $T$ converge to the correct one in the limit $
T\rightarrow\infty.$ However, for an arbitrary domain the aforementioned
conditions on the eigenvalues are difficult to verify. In \cite
{Stefanov-Yang}, the Dirichlet boundary condition was modified in such a way
that the computed function is the result of averaging of a family of
time-reversed solutions with different times $T$ (\textquotedblleft averaged
time reversal\textquotedblright). The authors showed that, in the case of
full boundary measurements, the operator that describes such a procedure is
a contraction. Therefore, the result can be used either as a crude
approximation, or a better solution can be computed by a converging Neumann
series (involving multiple solutions of the direct problem and averaged time
reversals). For the case when the data are available only on a part of the
boundary (i.e. for the \textquotedblleft partial data
problem\textquotedblright), theoretical analysis was not completed; however,
numerical experiments yielded successful reconstruction in this case, too.
In \cite{Acosta}, the problem with perfectly reflecting walls was considered
mostly from a theoretical standpoint, and its unique solvability was proven
under sufficiently general conditions.

In the present paper we propose a new time reversal technique in which the
Dirichlet boundary condition is replaced by a non-standard boundary
condition involving a linear combination of the normal- and time-
derivatives of the solution (see equation (\ref{E:Rev-p})). While this new
condition is satisfied by the solution of the direct problem, the error
resulting from the time reversal satisfies the wave equation with
dissipative boundary condition, and, thus, it decreases as time $t$
approaches $0$. We call this technique a \textquotedblleft dissipative time
reversal\textquotedblright. This method can be used either directly to
obtain good approximations to the initial pressure (which converges
exponentially as $T \to \infty$), or as a part of a Neumann series-based
iterative algorithm.\footnote{
Our idea of using the Neumann series-based iterative method is inspired by
the influential paper \cite{US}}. In addition to being efficient
numerically, our technique is relatively easy to analyze, allowing us to
deliver an intuitive proof for the convergences in both full and partial
data cases.

\section{Formulation of the problem and preliminaries}

\subsection{The direct problem}

Let $\Omega$ be a domain in $\mathbb{R}^{d}$ with smooth boundary $
\partial\Omega$. In the realistic setup of PAT and PAT, the dimension $d$
equals $3$. However, we will consider any $d \geq 2$, for the sake of
generality. We will denote by $\nu =\nu(x)$ the outward normal vector of $
\partial\Omega$ at $x$. The speed of sound $c(x)$ is a positive, infinitely
smooth function on $\mathbb{R}^d$.

For some positive time $T,$ let us consider the following mixed boundary
problem for the wave equation:
\begin{equation}
\left\{
\begin{array}{l}
u_{tt}(x,t)-c^{2}(x)\,\Delta u(x,t)=0,\quad(x,t)\in\Omega\times(0,T], \\
[4pt]
\frac{\partial}{\partial\nu}u(x,t)=0,\quad(x,t)\in\partial\Omega \times(0,T],
\\[4pt]
u(x,0)=u_{0}(x),~u_{t}(x,0)=u_{1}(x),\quad x\in\Omega,
\end{array}
\right.   \label{E:TATg}
\end{equation}
\newline
where $u_{0}$ and $u_{1}$ are arbitrary functions with $u_{0}\in
H_{0}^{1}(\Omega),$ $u_{1}\in L^{2}(\Omega).$ When $u_{1}\equiv0$ and $
u_{0}(x)$ coincides with the initial pressure $f(x)$ in the tissues, the
solution $u(x,t)$ of the above problem represents a model of TAT/PAT within
a resonant cavity formed by the reflecting walls (\cite
{Cox-Cavity,cox2008photo,Kunyansky-Cavity}). In this case, one attempts to
reconstruct initial pressure $f(x)$ from the measurements of $u(x,t)$ made
on some subset of boundary $\partial\Omega.$

For convenience, we consider here a slightly more general problem, where
both $u_{0}$ and $u_{1}$ are allowed to be non-zero. Let us assume that the
values of pressure $u(x,t)$ are measured on a part $\Gamma$ of the boundary $
\partial\Omega;$ we will consider both the case of full measurements when $
\Gamma$ coincides with $\partial\Omega,$ and the case of partial data when $
\Gamma$ is an open proper subset of $\partial\Omega$. We will denote the
measured data by $g(z,t):$%
\begin{equation}
g=u|_{\Gamma\times\lbrack0,T]}.   \label{E:def-g}
\end{equation}

Our goal is to solve the following problem:

\begin{problem}
\label{P:TATg} Find the pair $(u_{0},u_{1})$ from $g=u|_{\Gamma\times
\lbrack0,T]}$. In other words, invert the map
\begin{equation}
\Lambda_{T}:(u_{0},u_{1})\rightarrow g.   \label{E:map}
\end{equation}
\end{problem}

\subsection{Our approach to the inverse problem}

Let us fix a function $\lambda\in C^{\infty}(\partial\Omega)$ such that $
\lambda>0$ on $\Gamma$ and $\lambda=0$ on $\partial\Omega\setminus\overline{
\Gamma}.$ In addition, we will extend the data $g(z,t)$ to $\partial
\Omega\setminus\overline{\Gamma}\times\lbrack0,T]$ by zero. The main idea of
this paper is to find an approximation to $u_{0}$ and $u_{1}$ by solving the
following time reversal problem
\begin{equation}
\left\{
\begin{array}{l}
v_{tt}(x,t)-c^{2}(x)\,\Delta v(x,t)=0,\quad(x,t)\in\Omega\times\lbrack 0,T],
\\[4pt]
\frac{\partial}{\partial\nu}v(x,t)-\lambda(x)\,v_{t}(x,t)=-\lambda
(x)\,g_{t}(x,t),\quad(x,t)\in\partial\Omega\times\lbrack0,T], \\[4pt]
v(x,T)=0,~v_{t}(x,T)=0,\quad x\in\Omega.
\end{array}
\right.   \label{E:Rev-p}
\end{equation}

As we show below, for sufficiently large values of $T,$ the functions $v(x,0)
$ and $v_{t}(x,0)$ are good approximations to $u_{0}$ and $u_{1},$
correspondingly.

In particular, in the context of the inverse problem of TAT/PAT with $
u_{1}=0 $ and $u_{0}=f,$ the function $v(x,0)$ is a good approximation of $
f(x).$ One can either use this approximation directly, or to realize an
iterative refinement scheme converging to $f$ and corresponding to a
computation of a certain Neumann series$.$

The well-posedness of problem (\ref{E:Rev-p}), validity of a so-obtained
approximation and convergence of the Neumann series are the subject of the
following sections. The starting point, however, is the analysis of the
direct problem (\ref{E:TATg}).

\subsection{Properties of the direct problem}

Let us define the space of pairs of functions $\mathbb{H}$ as follows
\begin{equation*}
\mathbb{H}:=\{(u_{0},u_{1})|u_{0}\in H^{1}(\Omega),u_{1}\in L^{2}(\Omega)\},
\end{equation*}
where $H^{1}(\Omega)$ is the standard Sobolev space, with the norm defined,
for an arbitrary function $h(x)$ by
\begin{equation*}
\|h\|_{H^{1}(\Omega)}^{2}\equiv\int\limits_{\Omega}\left[ h^{2}(x)+\left|
\nabla h(x)\right| ^{2}\right] dx=\Vert
h\Vert_{L^{2}(\Omega)}^{2}+\Vert\nabla h\Vert_{L^{2}(\Omega)}^{2}.
\end{equation*}
Then, $\mathbb{H}$ is a Banach space under the norm $\Vert.\Vert$ defined by
\begin{equation*}
\Vert(u_{0},u_{1})\Vert^{2}=\Vert u_{0}\Vert_{H^{1}(\Omega)}^{2}+\Vert
c^{-1}u_{1}\Vert_{L^{2}(\Omega)}^{2}.
\end{equation*}
For $(u_{0},u_{1})\in\mathbb{H}$, we define
\begin{equation*}
\mathbb{E}(u_{0},u_{1})=\Vert\nabla u_{0}\Vert_{L^{2}(\Omega)}^{2}+\Vert
c^{-1}u_{1}\Vert_{L^{2}(\Omega)}^{2},
\end{equation*}
and the semi-norm $|.|$:
\begin{equation*}
|(u_{0},u_{1})|=[\mathbb{E}(u_{0},u_{1})]^{1/2}.
\end{equation*}
It is well known that for any $(u_{0},u_{1})\in\mathbb{H}$ equation~(\ref
{E:TATg}) has a unique solution $u$ lying within the following class $
\mathcal{K(}\Omega,T\mathcal{)}$:
\begin{equation*}
u\in\mathcal{K}(\Omega,T)\equiv\mathcal{C}([0,T];H^{1}(\Omega))~\cap ~
\mathcal{C}^{1}([0,T];L^{2}(\Omega)).
\end{equation*}
Therefore, $\Lambda_{T}$ (defined by (\ref{E:map})) is a well-defined map
from $\mathbb{H}$ to $\mathcal{C}([0,T];H^{1/2}(\partial \Omega))$.

\section{Well-posedness of a mixed boundary problem}

\label{S:Well} Let $\lambda\in C^{\infty}(\partial\Omega)$. In this section,
we study the following mixed boundary value problem:
\begin{equation}
\left\{
\begin{array}{l}
w_{tt}(x,t)-c^{2}(x)\,\Delta w(x,t)=0,\quad(x,t)\in\Omega\times(0,T), \\
[4pt]
\frac{\partial}{\partial\nu}w(x,t)+\lambda(x)\,w_{t}(x,t)=h_{t}(x,t),\quad
(x,t)\in\partial\Omega\times(0,T), \\[4pt]
w(x,0)=w_{0}(x),~w_{t}(x,0)=w_{1}(x),\quad x\in\Omega.
\end{array}
\right.   \label{E:Basic}
\end{equation}
Since this is not a standard problem, let us first define its weak solution.
We will follow the spirit of \cite{Bardos}, where the weak (distributional)
solution is defined for the case $h=0$. For our purposes, it is sufficient
to assume $(w_{0},w_{1})\in\mathbb{H}$ and $h\in\mathcal{C}
([0,T];H^{1/2}(\partial\Omega))$.

\begin{defi}
\label{D:Sol} Given $(w_{0},w_{1})\in\mathbb{H}$ and $h\in\mathcal{C}
([0,T];H^{1/2}(\partial\Omega))$, we say that $w\in\mathcal{D}
^{\prime}(\Omega\times(0,T))$ is a solution of (\ref{E:Basic}) if the
following equation holds for any test function $\psi\in\mathcal{D}
(\Omega\times(0,T))$:
\begin{align}
\left\langle w,\psi\right\rangle & =-
\iint_{\Omega\times\lbrack0,T]}c^{-2}(x)\,\big[w_{0}(x)\,
\Psi_{t}(x,0)-w_{1}(x)\,\Psi(x,0)\big]\,dx+\int_{\partial\Omega}\lambda(x)
\,w_{0}(x)\,\Psi(x,0)\,dx  \notag \\
&
-\iint_{\partial\Omega\times\lbrack0,T]}h(x,t)\Psi_{t}(x,t)\,dx\,dt-\int_{
\partial\Omega}h(x,0)\,\Psi(x,0)\,dx.   \label{E:weaksol}
\end{align}
Here, $\Psi$ is the solution of the dual problem
\begin{equation*}
\left\{
\begin{array}{l}
c^{-2}(x)\,\Psi_{tt}(x,t)-\Delta\Psi(x,t)=\psi(x,t),\quad(x,t)\in\Omega
\times(0,T), \\[4pt]
\partial_{\nu}\Psi(x,t)-\lambda(x)\,\Psi_{t}(x,t)=0,\quad(x,t)\in
\partial\Omega\times(0,T), \\[4pt]
\Psi(x,T)=0,~\Psi_{t}(x,T)=0,\quad x\in\Omega.
\end{array}
\right.
\end{equation*}
\end{defi}

The uniqueness of solution $w$ is obvious, since $\left\langle w,\psi
\right\rangle $ is defined by the right hand side of (\ref{E:weaksol}) for
all test functions $\psi.$ The existence and regularity of this solution are
more complicated. We only present here a partial result that is needed in
the present paper:

\begin{prop}
\label{P:Well} Assume that $h=0$ and $(w_{0},w_{1})\in\mathbb{H}$. Then,
problem (\ref{E:Basic}) has a unique solution
\begin{equation*}
w\in\mathcal{K}(\Omega,T).
\end{equation*}

\end{prop}

Let us recall a result by \cite{ikawa1970mixed}. Assume that $h =0$ and $
(w_{0},w_{1})\in H^{2}(\Omega)\times H^{1}(\Omega)$ satisfies the
compatibility condition
\begin{equation}
\frac{\partial}{\partial\nu}w_{0}+\lambda(x)\,w_{1}=0,\quad\mbox{ on }
\partial \Omega.   \label{E:comp1}
\end{equation}
Then, problem (\ref{E:Basic}) has a unique solution
\begin{equation*}
w\in\mathcal{C}([0,T]; H^{2}(\Omega))\cap\mathcal{C}^{1}([0,T];H^{1}(
\Omega))\cap\mathcal{C}^{2}([0,T];L^{2}(\Omega)).
\end{equation*}
Proposition~\ref{P:Well} can be proved by a simple approximation argument
which we present here for the sake of completeness.

\begin{proof}[\textbf{Proof of Proposition~\protect\ref{P:Well}}]
It suffices to prove the existence part, since the uniqueness is trivial as
mentioned above. Consider $(\varphi_{0}^{n},\varphi_{1}^{n}) \in
H^{2}(\Omega) \times H^{1}(\Omega)$ such that:

\begin{itemize}
\item[1)] $(\varphi_{0}^{n},\varphi_{1}^{n})\rightarrow(w_{0},w_{1})$ in $
H^{1}(\Omega)\times L^{2}(\Omega)$,

\medskip

\item[2)] $\partial_{\nu}\varphi^{n}_{0}|_{\partial \Omega}=0$ and $
\varphi_{1}^{n}|_{\partial\Omega}=0$.
\end{itemize}

Such a sequence $\{(\varphi_{0}^{n},\varphi_{1}^{n})\}$ always exists. For
instance, we can choose $\varphi_{0}^{n}$ to be a linear combination of
eigenvectors of the Neumann Laplacian. Meanwhile, $\varphi_{1}^{n}$ can be
chosen as a linear combination of eigenvectors of the Dirichlet Laplacian.

Consider problem (\ref{E:Basic}) with the initial condition $
(\varphi_{0}^{n},\varphi_{1}^{n})$ (instead of $(w_{0},w_{1})$). Since $
(\varphi_{0}^{n},\varphi_{1}^{n})$ satisfies the compatibility condition (
\ref{E:comp1}), problem (\ref{E:Basic}) with such initial condition has a
unique solution
\begin{equation*}
w^{n}\in\mathcal{C}([0,T]; H^{2}(\Omega))\cap\mathcal{C}^{1}([0,T];
H^{1}(\Omega))\cap\mathcal{C}^{2}([0,T];L^{2}(\Omega)).
\end{equation*}
Moreover, by simple integration by parts, we obtain for any $t_{0}\in
\lbrack0,T]$:
\begin{equation*}
\mathbb{E}(w^{n}-w^{m},t_{0})+\int_{\partial\Omega\times\lbrack0,t_{0}]}
\lambda(x)\,|w_{t}^{n}-w^{m}_{t}|^{2}\,dx\,dt=\mathbb{E}(w^{n}-w^{m},0).
\end{equation*}
Therefore,
\begin{equation*}
\mathbb{E}(w^{n}-w^{m},t_{0})\leq\mathbb{E}(w^{n}-w^{m},0).
\end{equation*}
We notice that
\begin{equation*}
\mathbb{E}(w^{n}-w^{m},0)=\Vert\nabla(\varphi_{0}^{n}-\varphi_{0}^{m})
\Vert_{L^{2}(\Omega)}^{2}+\Vert
c^{-1}(\varphi_{1}^{n}-\varphi_{1}^{m})\Vert_{L^{2}(\Omega)}^{2}
\rightarrow0.
\end{equation*}
Since
\begin{equation*}
\mathbb{E}(w^{n}-w^{m},t)=\Vert\nabla(w^{n}(t,.)-w^{m}(t,.))\Vert^{2}+\Vert
c^{-1}(w_{t}^{n}(t,.)-w_{t}^{m}(t,.))\Vert^{2},
\end{equation*}
we obtain $\{w^{n}\}$ is a Cauchy sequence in $\mathcal{C}
([0,T];H^{1}(\Omega))$ and $\{w_{t}^{n}\}$ is a Cauchy sequence in $\mathcal{
C}([0,T];L^{2}(\Omega))$. Therefore, there exists a function $w\in \mathcal{K
}(\Omega,T)$ such that
\begin{equation*}
\{w^{n}\}\rightarrow w\mbox{ in }\mathcal{C}([0,T],H^{1}(\Omega)),\quad
\mbox{ and }\{w_{t}^{n}\}\rightarrow w_{t}\mbox{ in }\mathcal{C}
([0,T];L^{2}(\Omega)).
\end{equation*}
It is easy to verify that $w$ is a solution of problem (\ref{E:Basic}), from
the definition~\ref{D:Sol}. This finishes our proof for the existence.
\end{proof}

Let us apply the above result to the time reversal problem (\ref{E:Rev-p}).

\begin{prop}
Problem (\ref{E:Rev-p}) has a unique solution $v\in\mathcal{K}(\Omega,T).$
\end{prop}

\begin{proof}
We first notice that problem (\ref{E:Rev-p}) is the time reversed version of
problem (\ref{E:Basic}), considered in Section~\ref{S:Well}. The uniqueness
of the solution follows trivially from the definition~\ref{D:Sol}. We now
prove the existence. Let us consider the following problem
\begin{equation}
\left\{
\begin{array}{l}
U_{tt}(x,t)-c^{2}(x)\,\Delta U(x,t)=0,\quad(x,t)\in\Omega\times\lbrack 0,T],
\\[4pt]
\frac{\partial}{\partial\nu}U(x,t)-\lambda(x)\,U_{t}(x,t)=0,\quad
(x,t)\in\partial\Omega\times\lbrack0,T], \\[4pt]
U(x,T)=u(x,T),~U_{t}(x,T)=u_{t}(x,T),\quad x\in\Omega.
\end{array}
\right.   \label{E:err}
\end{equation}
where $u(x,t)$ is the solution of the direct problem (\ref{E:TATg}). Since $
(u(.,T),u_{t}(.,T))\in\mathbb{H}$, we obtain from Proposition~\ref{P:Well}
that the above problem has a (unique) solution $U\in\mathcal{K}(\Omega,T).$

We notice that $u$ also satisfies the boundary condition in problem (\ref
{E:Rev-p}).

Let $v=u-U$. It is easy to verify that $v$ is a solution of (\ref{E:Rev-p}).
Moreover, from the regularity of $u$ and $U$, we conclude that $v\in
\mathcal{K}(\Omega,T).$
\end{proof}

\section{Solution of the inverse problem}

\subsection{Contraction properties of the time reversal operator}

Similarly to \cite{Stefanov-Yang,Acosta}, our analysis on the inversion of $
\Lambda $ relies on known results on stabilization of waves \cite{Bardos}.
For the sake of simplicity, we assume that all the geodesics of $(\mathbb{R}
^{3},c^{-2}\,dx^{2})$ have finite contact order with the boundary $\partial
\Omega $. Under this condition, the generalized bi-characteristics of the
wave operator $\Box =\partial _{tt}-c^{2}(x)\,\Delta $ on $\overline{\Omega }
$ are uniquely defined (see, e.g., \cite{Bardos}). Their projections on the
physical space (i.e., $\overline{\Omega }$) are called the generalized rays.

\medskip

Throughout the paper, we will assume that the following condition is
satisfied:

\begin{condition}
\label{A:Gcc} There is a finite value $T(\Omega,\Gamma)>0$ such that every
generalized ray of length \footnote{
The length here is understood in the metric $c^{-2} \, dx^2$.
Condition~\ref{A:Gcc} means that all the singularities of the  solution of
the wave equation (\ref{E:TATg}) that start propagating at time $0$, traveling along
generalized bi-characteristics, reach the set $\Gamma$ within the
time interval $[0,T(\Omega,\Gamma)]$.}
$T(\Omega,\Gamma)$
intersects $\Gamma$ at one (or more) non-diffractive point(s).
\end{condition}

This is the \textbf{geometric control condition} (GCC), well known in
control theory. We refer the reader to \cite{Bardos} for a detailed
discussion of GCC. This condition was shown in \cite{Stefanov-Yang,Acosta}
to be sufficient for the stability of the inversion of $\Lambda$. Our
results are also based on the assumption that GCC is satisfied.

\medskip

Our work is based on the following fundamental result due to \cite{Bardos}:

\begin{prop}
\label{P:contract} Consider problem (\ref{E:Basic}) with $g=0$. Assume that
Condition~\ref{A:Gcc} holds. Then, there is $\delta(T)<1$ such that
\begin{equation*}
|(u(.,T),u_{t}(.,T))|\leq\delta(T)\,|(u_{0},u_{1})|.
\end{equation*}
\end{prop}

\begin{remark}
\label{R:exp}By applying Proposition \ref{P:contract} $k$ times, one obtains
\begin{equation*}
|(u(.,kT),u_{t}(.,kT))|\leq\delta^{k}(T)\,|(u_{0},u_{1})|.
\end{equation*}
It is easy to conclude that $|(u(.,T),u_{t}(.,T))|\rightarrow0$
exponentially as $T\rightarrow\infty$. That is, there are constant $
C_{1}(\Omega)$ and $a>0$ such that
\begin{equation*}
|(u(.,T),u_{t}(.,T))|\leq C_{1}(\Omega)e^{-a T}\,|(u_{0},u_{1})|.
\end{equation*}
\end{remark}

Given boundary data $g(x,t)$, we define the time reversal operator $A$
through the solution $v(x,t)$ of the problem (\ref{E:Rev-p}):
\begin{equation*}
Ag=(v(.,0),v_{t}(.,0)).
\end{equation*}

\begin{lemma}
\label{T:Main-p} Assume that condition\textbf{~\ref{A:Gcc}} holds and $T
\geq T(\Omega,\Gamma)$. Let $I$ denote the identity map, then there is $
\delta(T)<1$ such that
\begin{equation*}
|(I-A\Lambda_{T})(u_{0},u_{1})|\leq\delta(T)|(u_{0},u_{1})|, \quad \mbox{
for all } (u_0,u_1) \in \mathbb{H}.
\end{equation*}
In other words, $I-A\Lambda_{T}$ is a contraction on $\mathbb{H}$ under the
semi norm $|\cdot|$.
\end{lemma}

\begin{proof}
Let $(u_{0},u_{1})\in \mathbb{H}$ and $u$ be the solution of (\ref{E:TATg}).
We observe that
\begin{equation*}
(I-A\Lambda _{T})(u_{0},u_{1})=(U(.,0),U_{t}(.,0)),
\end{equation*}
where $U$ is the solution of the problem (\ref{E:err}).

Due to the well-known conservation of energy in the solution of the wave
equation with Neumann boundary condition,
\begin{equation*}
\mathbb{E}(u_{0},u_{1})=\mathbb{E}(u(.,T),u_{t}(.,T)),
\end{equation*}
or
\begin{equation*}
|(u_{0},u_{1})|=|(u(.,T),u_{t}(.,T))|.
\end{equation*}
Due to Proposition~\ref{P:contract},
\begin{equation*}
|(U(.,0),U_{t}(.,0))|\leq\delta(T)\,|(u(.,T),u_{t}(.,T))|,
\end{equation*}
and the proof follows.
\end{proof}

\begin{remark}
\label{R:Exp-p} Using Remark \ref{R:exp}, instead of Proposition~\ref
{P:contract}, in the above proof, we obtain that there exist constants $
C_{1}(\Omega)$ and $a>0$ such that
\begin{equation*}
|(I-A\Lambda_{kT})(u_{0},u_{1})|\leq C_{1}(\Omega)e^{-aT}|(u_{0},u_{1})|.
\end{equation*}
That is, the\ induced seminorm of $I-A\Lambda_{T}$ decreases exponentially
as $T\rightarrow\infty$.
\end{remark}

Lemma \ref{T:Main-p} describes the contraction property of our time reversal
operator in the semi-norm $|.|$. By itself, this result is not sufficient to
prove the convergence under the norm $\|.\|.$ However, by restricting our
attention to appropriate subspaces of $\mathbb{H}$,  such a convergence can
be attained. Indeed, let us consider the subspace $\mathbb{H}_{0}$ of $
\mathbb{H}$ defined by
\begin{equation*}
\mathbb{H}_{0}\equiv\left\{ \mathbf{h}=(h_{0,}h_{1})\in\mathbb{H}\
\left\vert \ {\int\limits_{\partial\Omega}h_{0}\,dx=0}\right. \right\} ,
\end{equation*}
and the subspace $\mathbb{H}_{1}$ of $\mathbb{H}_{0}$ consisting of pairs
with the second component equal to zero:
\begin{equation*}
\mathbb{H}_{1}\equiv\left\{ \mathbf{h}=(h_{0,}0)\in\mathbb{H}_{0}\right\} .
\end{equation*}

Let us introduce two projectors $\Pi_{0}$ and $\Pi_{1}$ mapping the elements
of $\mathbb{H}$ into $\mathbb{H}_{0}$ and $\mathbb{H}_{1},$ respectively:
\begin{align*}
\Pi_{0}\mathbf{h} & \equiv\Pi_{0}(h_{0,}h_{1})\equiv\left( h_{0}-\frac {1}{
|\partial\Omega|}{\int\limits_{\partial\Omega}h_{0},h_{1}}\right) {,} \\
\Pi_{1}\mathbf{h} & \equiv\Pi_{1}(h_{0,}h_{1})\equiv\left( h_{0}-\frac {1}{
|\partial\Omega|}{\int\limits_{\partial\Omega}h_{0},0}\right) {,}
\end{align*}
where $|\partial\Omega|$ is the surface area of $\partial\Omega$. These
projectors do not increase the semi-norm $|.|$. That is,
\begin{equation*}
|\Pi_{0}\mathbf{h|\leq}|\mathbf{h|,}\quad|\Pi_{1}\mathbf{h|\leq}|\mathbf{h|}
.
\end{equation*}
Moreover, the subspaces $\mathbb{H}_{0}$ and $\mathbb{H}_{1}$ are invariant
under compositions $\Pi_{0}(I-A\Lambda_{T})$ and $\Pi_{1}(I-A\Lambda_{T}).$
In addition,
\begin{align*}
\Pi_{0}(I-A\Lambda_{T})\mathbf{h} & \mathbf{=}(I-\Pi_{0}A\Lambda _{T})
\mathbf{h,\qquad\forall h}\text{ }\in\mathbb{H}_{0}, \\
\Pi_{1}(I-A\Lambda_{T})\mathbf{h} & \mathbf{=}(I-\Pi_{1}A\Lambda _{T})
\mathbf{h,\qquad\forall h}\text{ }\in\mathbb{H}_{1}.
\end{align*}
Therefore, in accordance with Lemma \ref{T:Main-p}, operators $(I-\Pi
_{0}A\Lambda_{T})$ and $(I-\Pi_{1}A\Lambda_{T})$ are contractions in $
\mathbb{H}_{0}$ and $\mathbb{H}_{1}$ correspondingly, under the seminorm $
|.|:$
\begin{align}
\mathbf{|}(I-\Pi_{0}A\Lambda_{T})\mathbf{h|} & \mathbf{\leq|}(I-A\Lambda
_{T})\mathbf{h}|\leq\delta(T)\mathbf{|h|,\qquad\forall h}\text{ }\in \mathbb{
H}_{0},  \label{E:seminorm0} \\
\mathbf{|}(I-\Pi_{1}A\Lambda_{T})\mathbf{h|} & \mathbf{\leq\mathbf{|}}
(I-A\Lambda_{T})\mathbf{\mathbf{h|}}\leq\delta(T)\mathbf{|h|,\qquad\forall h}
\text{ }\in\mathbb{H}_{1}.   \label{E:seminorm1}
\end{align}
Now we can invert operator $\Lambda_{T}$, restricted to $\mathbb{H}_{0}$, by
constructing converging Neumann series
\begin{equation*}
\sum_{k=0}^{\infty}(I-\Pi_{0}A\Lambda_{T})^{k}\Pi_{0}A
\Lambda_{T}(u_{0},u_{1})=\sum_{k=0}^{\infty}(I-\Pi_{0}A\Lambda_{T})^{k}
\Pi_{0}Ag.
\end{equation*}
Below we show that this series converges not only under the semi-norm $|.|$
but also in the norm $\|.\|.$ In other words, we claim that the partial sums
$\mathbf{u}^{(n)}$of these series
\begin{equation}
\mathbf{u}^{(k)}=\sum_{j=0}^{k}(I-\Pi_{0}A\Lambda_{T})^{j}\Pi_{0}Ag
\label{E:Neumann-partial}
\end{equation}
converge to $\mathbf{u}=(u_{0},u_{1})$ in the norm $\|.\|.$ These partial
sums can be easily computed by the following iterative algorithm
\begin{align}
\mathbf{u}^{(0)} & =0,  \notag \\
\mathbf{u}^{(k+1)} & =(I-\Pi_{0}A\Lambda_{T})\mathbf{u}^{(k)}+\Pi_{0}Ag.
\label{E:iterations-general}
\end{align}

\medskip

The following theorem gives us a solution to Problem~\ref{P:TATg} by a
Neumann series:

\begin{theorem}
\label{T:Neumann-ser} Suppose that condition\textbf{~\ref{A:Gcc}} holds and
the observation time $T$ satisfies $T\geq T(\Omega,\Gamma).$ Then, the
iterations $\mathbf{u}^{(k)}$ defined by (\ref{E:Neumann-partial}) (or,
equivalently, by (\ref{E:iterations-general})) converge to $\mathbf{u}$ in
norm $\|.\|$, as follows:
\begin{equation*}
\Vert\mathbf{u}-\mathbf{u}^{(k)}\Vert\leq C_{P}(\Omega)\,\delta(T)^{k}\Vert
\mathbf{u}\Vert, \quad \mbox{ for all } \mathbf{u} \in \mathbb{H}_0,
\end{equation*}
with some constant $C_{P}(\Omega)>1$ specified below.
\end{theorem}

To prove the above theorem, we will need the following generalization of the
Poincare inequality.

\begin{lemma}
\label{L:Poincare} There is a constant $C_{P}(\Omega)>1$ such that for all $
h\in H^{1}(\Omega)$ satisfying $\int_{\partial\Omega}h\,dx=0$ the following
inequality holds:
\begin{equation*}
\Vert h\Vert_{H^{1}(\Omega)}\leq C_{P}(\Omega)\Vert\nabla
h\Vert_{L^{2}(\Omega)}.   \label{E:Poincare}
\end{equation*}
\end{lemma}

\begin{proof}
{Indeed, assume that the above statement is not true. Then, there exists a
sequence $\{h_{n}\}_{n=1}^{\infty}\subset H^{1}(\Omega)$ such that
\begin{equation*}
\Vert h_{n}\Vert_{L^{2}(\Omega)}=1\quad\text{and}\quad\lim_{n\rightarrow
\infty}\Vert\nabla h_{n}\Vert_{L^{2}(\Omega)}=0.
\end{equation*}
It follows that $\{h_{n}\}$ is bounded in $H^{1}(\Omega).$ Therefore, there
is a weakly converging in $H^{1}(\Omega)$ subsequence $\{h_{n_{k}}\}_{k=1}^{
\infty}$ and function }$u\in${\ $H^{1}(\Omega),$ such that,
\begin{align*}
h_{n_{k}} & \rightarrow u\mbox{ strongly in }L^{2}(\Omega), \\
h_{n_{k}} & \rightarrow u\mbox{ weakly in
}H^{1}(\Omega), \\
h_{n_{k}} & \rightarrow u\mbox{ weakly in }H^{1/2}(\partial\Omega).
\end{align*}
The limit }$u$ has the following three properties{\
\begin{equation*}
\Vert u\Vert_{L^{2}(\Omega)}=1,\quad\Vert\nabla
u\Vert_{L^{2}(\Omega)}=0,\quad\int\limits_{\partial\Omega}u\,dx=0.
\end{equation*}
The last two equations yield $u\equiv0$, which contradicts to the first
property }$\Vert u\Vert_{L^{2}(\Omega)}=1${, thus completing the proof.}
\end{proof}

\begin{corollary}
\label{C:cor} There is a constant $C_{P}(\Omega)\,>1$ (given by the above
Lemma) such that the following inequality holds:
\begin{equation}
|\mathbf{h}|\leq\Vert\mathbf{h}\Vert\leq C_{P}(\Omega)|\mathbf{h}|, \quad
\mbox{ for all } \mathbf{h} \in \mathbb{H}_0.   \label{E:Corollary}
\end{equation}
\end{corollary}

\begin{proof}[\textbf{Proof of Theorem~\protect\ref{T:Neumann-ser}}]
By a simple induction argument applied to the recurrence relation (\ref
{E:iterations-general}) one obtains the following identity:
\begin{equation*}
\mathbf{u}-\mathbf{u}^{(k)}=(I-\Pi_{0}A\Lambda_{T})^{k}\mathbf{u},\quad
k=0,1,2,3,...
\end{equation*}
Since $\mathbf{u}\in\mathbb{H}_{0}$, applying inequality (\ref{E:seminorm0})
results in the following inequality
\begin{equation*}
|\mathbf{u}-\mathbf{u}^{(k)}|\leq\delta(T)^{k}|\mathbf{u}|.
\end{equation*}
Further, using inequalities (\ref{E:Corollary}), one can transition to the
following norm estimate
\begin{equation*}
\|\mathbf{u}-\mathbf{u}^{(k)} \| \leq C_{P}(\Omega)\delta(T)^{k} \|\mathbf{u}
\|,
\end{equation*}
thus, proving convergence of the Neumann series in the norm $\|.\|.$
\end{proof}

\begin{remark}
\label{myremark}We note that, due to Remark~\ref{R:Exp-p},
\begin{equation}
\Vert\mathbf{u}-\Pi_{0}Ag\Vert\leq C_{1}(\Omega)C_{P}(\Omega)e^{-aT}\Vert
\mathbf{u}\Vert.
\end{equation}
Therefore, when $T$ is sufficiently large, $\mathbf{u}^{(1)}=\Pi_{0}Ag$ is a
good approximation to $\mathbf{u}$.
\end{remark}

\subsection{Inverse problem of TAT/PAT}

The inverse problem of TAT/PAT is a particular case of Problem~\ref{P:TATg}
with $u_{0}=f$ and $u_{1}=0.$ Therefore, $f$ can be recovered from $g$ by
using the algorithm described in the previous section. However, the
convergence can be accelerated and computations simplified by projecting
computed approximations onto space $\mathbb{H}_{1}$ rather than $\mathbb{H}
_{0}.$

Let us introduce the notation $\mathbf{f}=(f,0).$ The measured data $g$ are
still defined by equation (\ref{E:def-g}) with $u(x,t)$ being a solution of
the direct problem (\ref{E:TATg}) with initial conditions
\begin{equation*}
(u(0,x),u_{t}(0,x))=\mathbf{f}(x).
\end{equation*}
We show below that the partial sums
\begin{equation*}
\mathbf{u}^{(k)}=\sum_{j=0}^{k}(I-\Pi_{1}A\Lambda_{T})^{j}\Pi_{1}Ag;
\end{equation*}
of the following Neumann series converge to $\mathbf{f}$ in norm $\|.\|$.
These sums are easy to compute using the following iterative relation
\begin{align}
\mathbf{u}^{(0)} & =0,  \notag \\
\mathbf{u}^{(k+1)} & =(I-\Pi_{1}A\Lambda_{T})\mathbf{u}^{(k)}+\Pi_{1}Ag.
\label{E:final-alg}
\end{align}

\begin{theorem}
\label{T:TAT} Assume that the initial conditions of Problem ~\ref{P:TATg} is
given by $(u_{0},u_{1})=\mathbf{f}.$ Suppose also condition\textbf{~\ref
{A:Gcc}} is satisfied and $T\geq T(\Omega,\Gamma).$ Then, iterations $
\mathbf{u}^{(k)}$ defined by (\ref{E:Neumann-partial}) (or, equivalently, by
(\ref{E:iterations-general})) converge to $\mathbf{f}$ in norm $\|.\|$, as
follows
\begin{equation*}
\Vert\mathbf{f}-\mathbf{u}^{(k)}\Vert\leq C_{P}(\Omega)\,\delta(T)^{k}\Vert
\mathbf{f}\Vert.
\end{equation*}
\end{theorem}

\begin{proof}
The proof is almost identical to that of Theorem \ref{T:Neumann-ser}, with
inequality (\ref{E:seminorm1}) used instead of (\ref{E:seminorm0}).
\end{proof}

\begin{remark}
Similarly to Remark~\ref{myremark},
\begin{equation*}
\Vert\mathbf{f}-\Pi_{1}Ag\Vert\leq C_{1}(\Omega)C_{P}(\Omega)e^{-aT}\Vert
\mathbf{f}\Vert,
\end{equation*}
and, when $T$ is sufficiently large, $\mathbf{u}^{(1)}=\Pi_{1}Ag$ is a good
approximation to $\mathbf{f}$.

\end{remark}

\section{Numerical implementation and simulations}

\begin{figure}[t]
\begin{center}
\begin{tabular}{ccc}
\includegraphics[width=2.0in,height=2.0in]{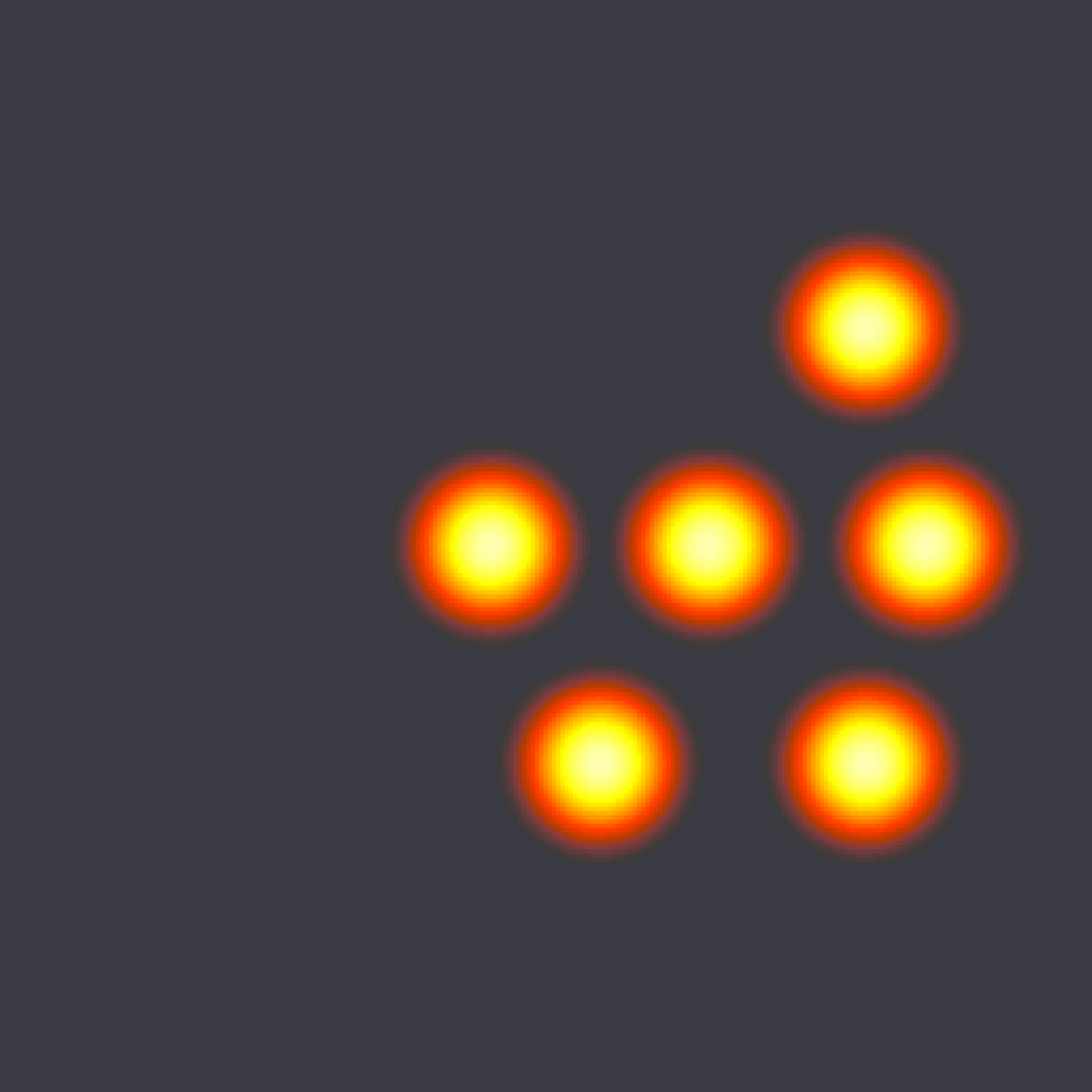} &
\includegraphics[width=2.0in,height=2.0in]{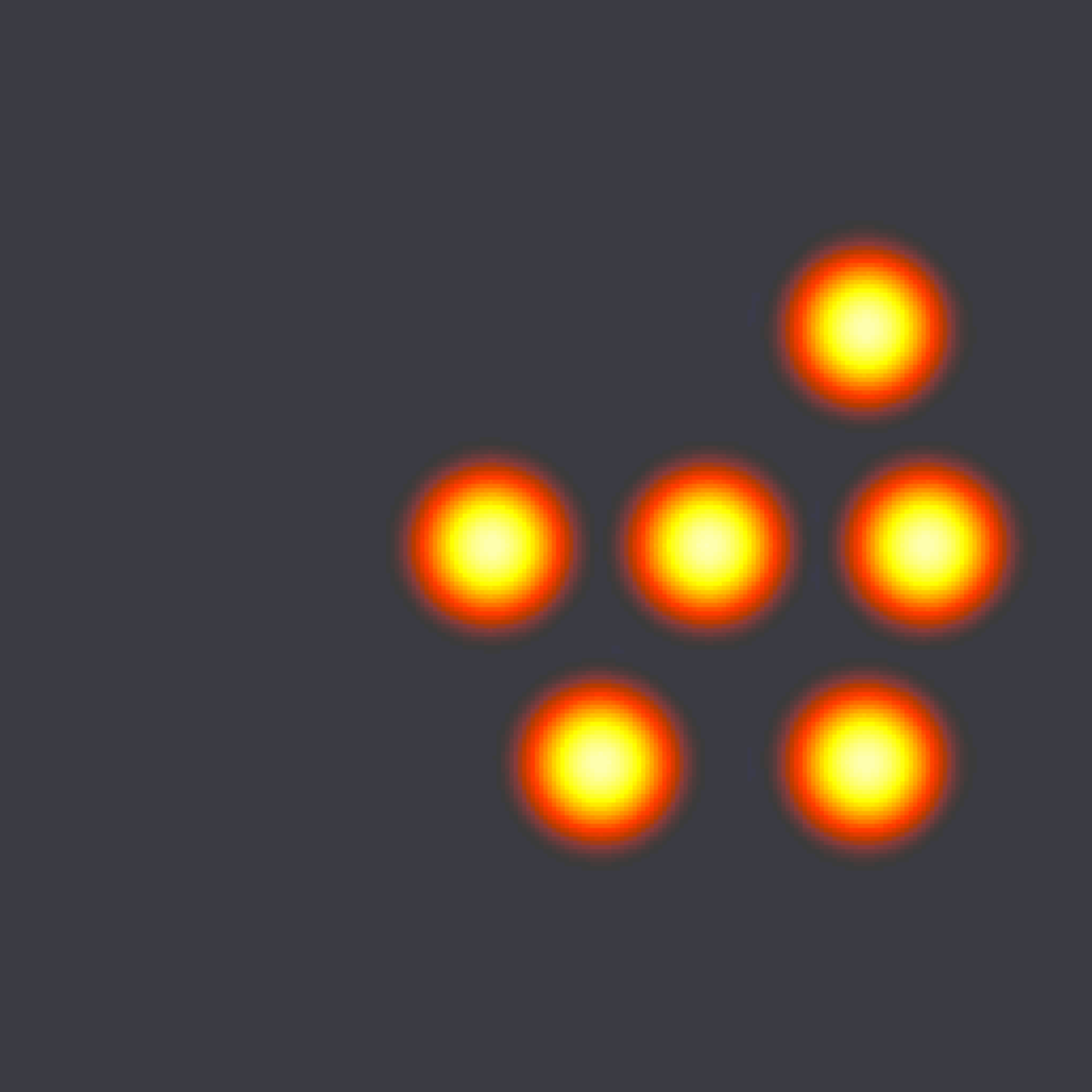} &
\includegraphics[width=2.0in,height=2.0in]{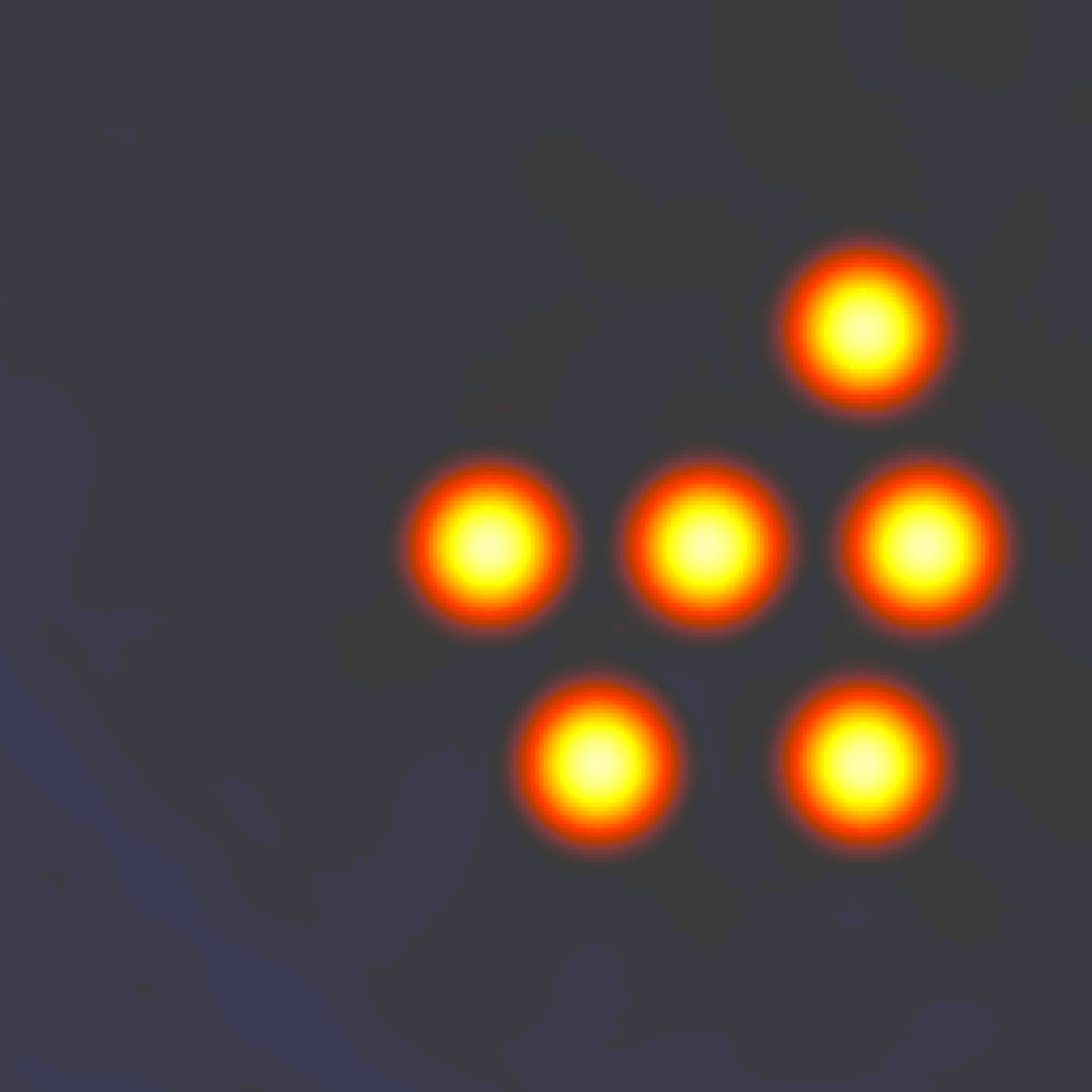} \\
(a) Phantom & (b) Full data reconstruction & (c) Partial data reconstruction
\\
\phantom{a} &  &  \\
&  &
\end{tabular}
\\[0pt]
\includegraphics[width=4.6in,height=1.5in]{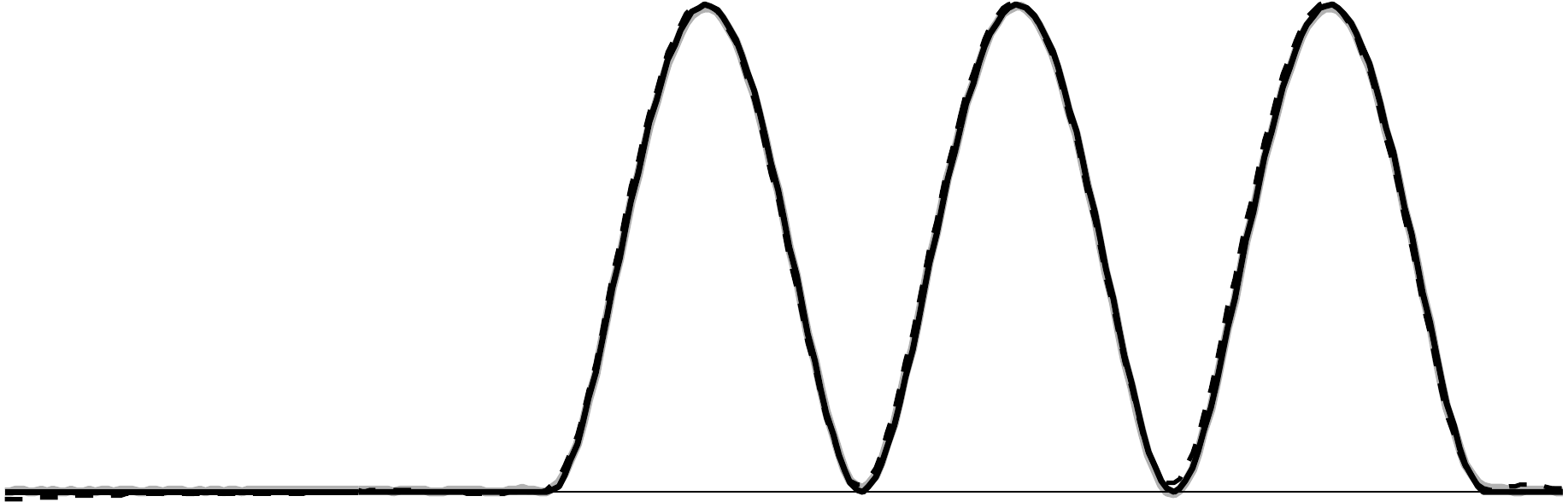}\\[0pt]
(d) Central horizontal cross-sections of (a), (b), and (c)\\[0pt]
\end{center}
\caption{Reconstruction with $T=5$. In (d): gray line represents the
phantom, dashed line shows image reconstructed from the partial data, black
line shows full data reconstruction}
\label{F:t5}
\end{figure}

\subsection{Implementation}

\begin{figure}[t]
\begin{center}
\includegraphics[width=4.6in,height=1.5in]{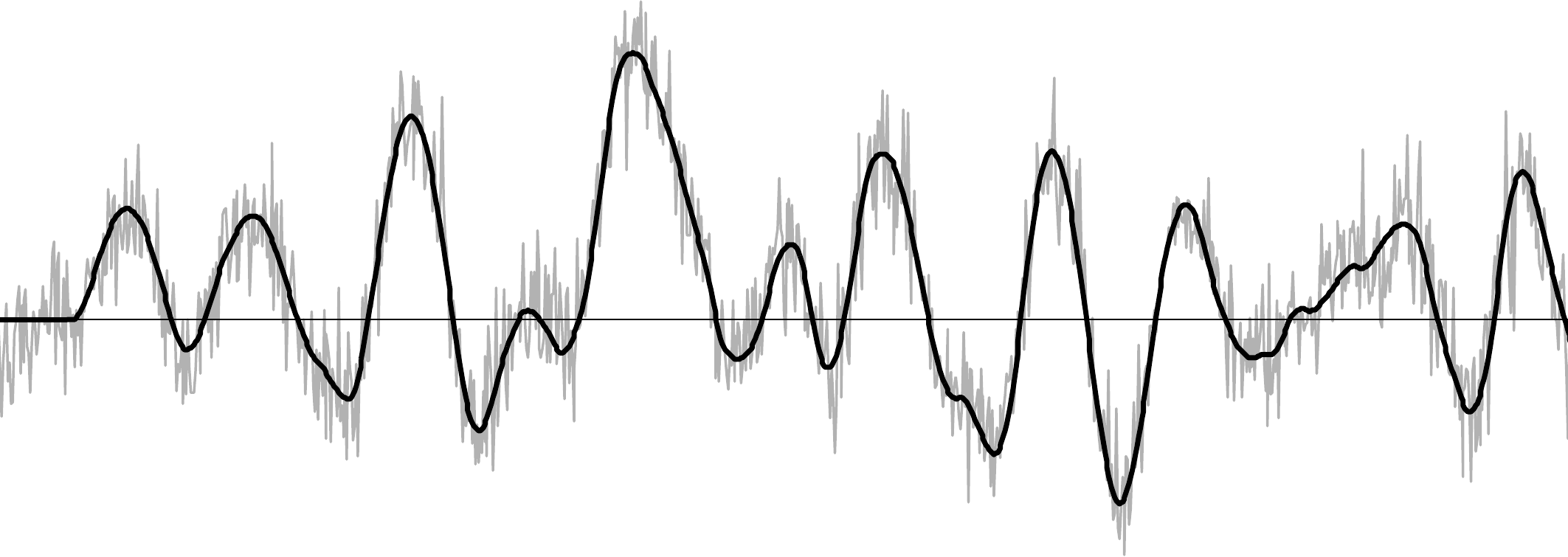}\\[0pt]
(a) Exact data $g(x,t)$ and data with added 50\% noise (in $L^{2}$ norm),
for one $x$ \\[0.5cm]
\par
\begin{tabular}{ccc}
\includegraphics[width=2.0in,height=2.0in]{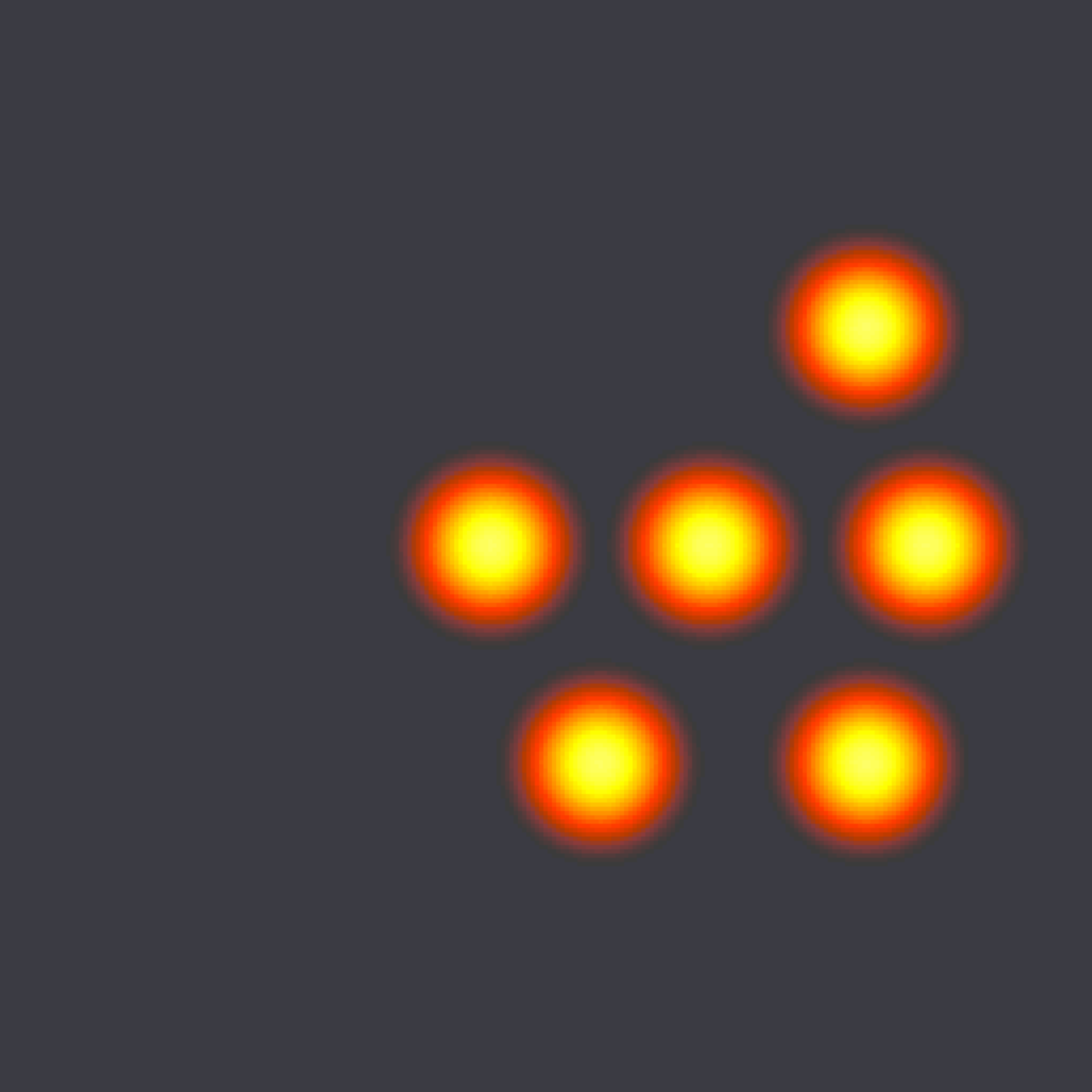} &
\includegraphics[width=2.0in,height=2.0in]{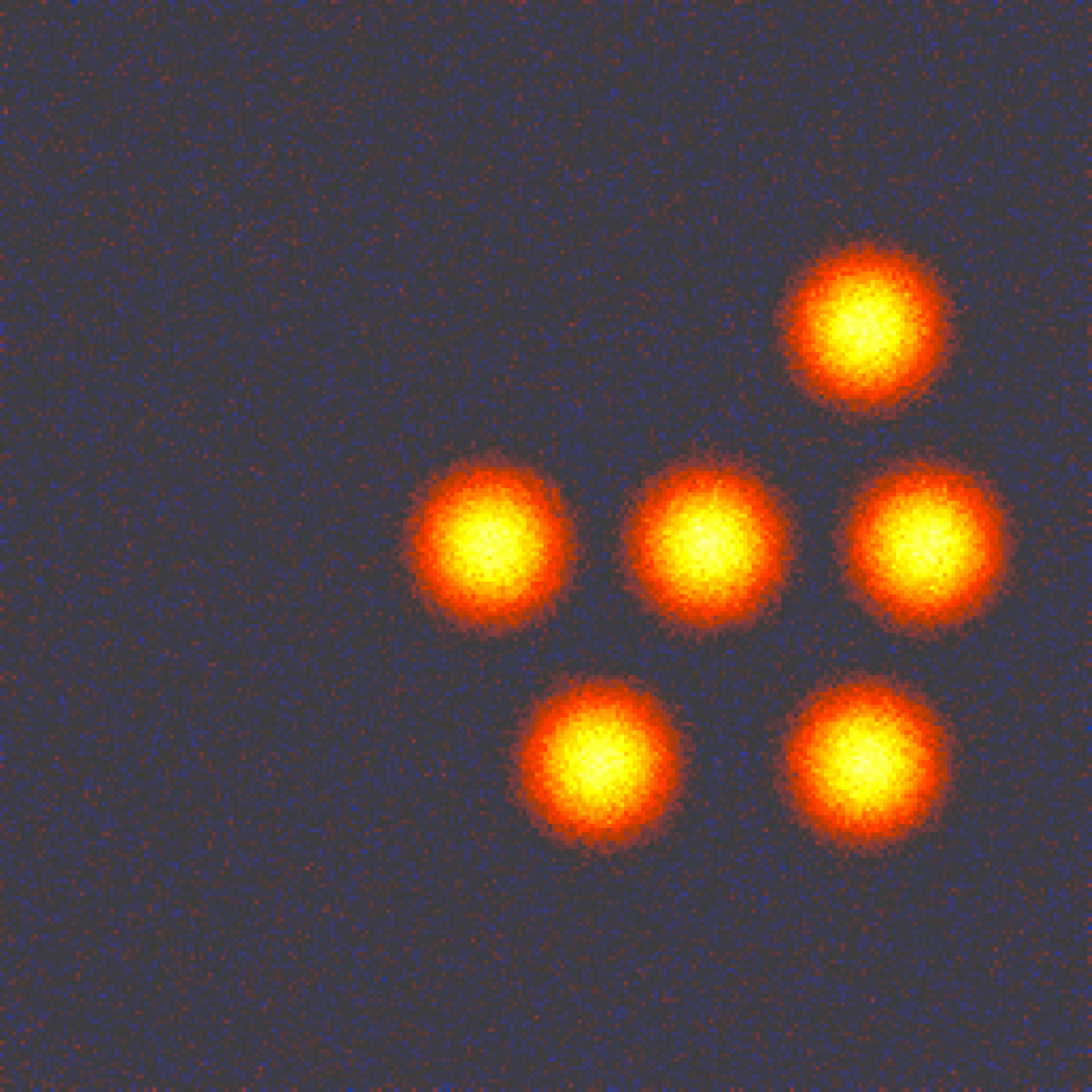} &
\includegraphics[width=2.0in,height=2.0in]{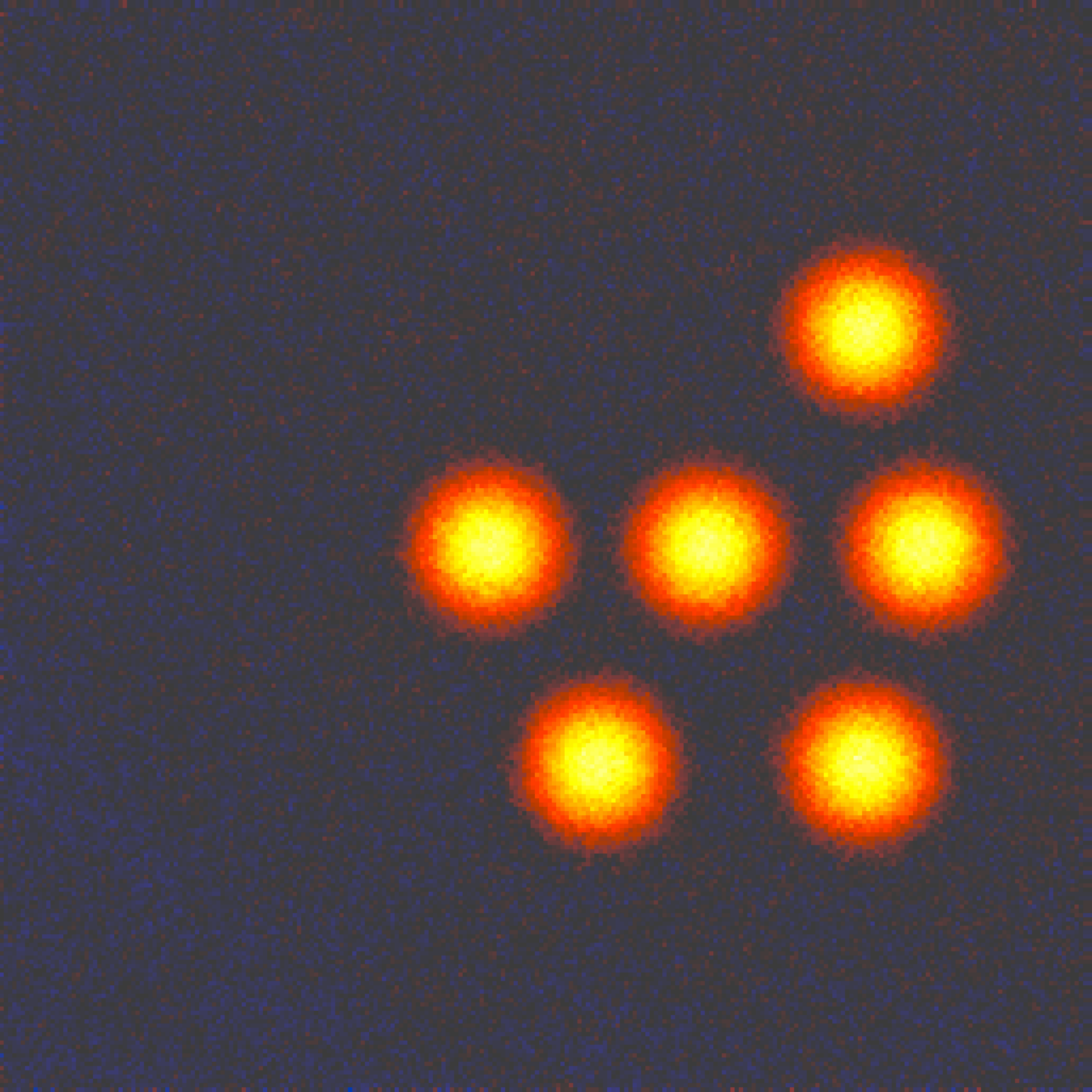} \\
(b) Phantom & (c) Full data reconstruction & (d) Partial data reconstruction
\\
\phantom{a} &  &  \\
&  &
\end{tabular}
\\[0pt]
\includegraphics[width=4.6in,height=1.5in]{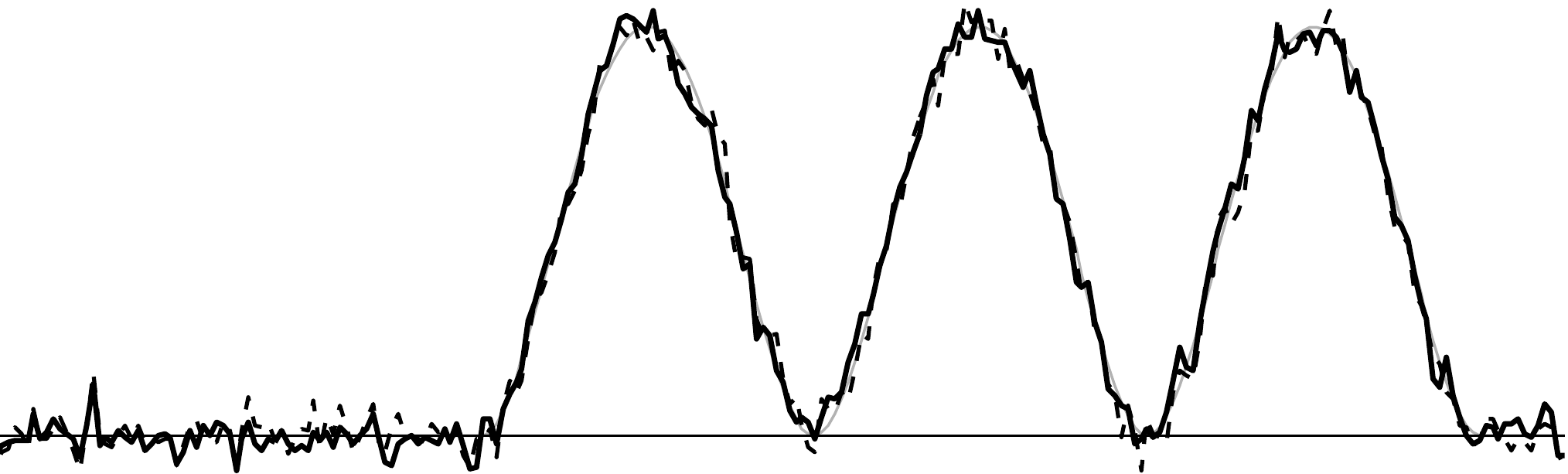}\\[0pt]
(e) Central horizontal cross-sections of (b), (c), and (d)\\[0pt]
\end{center}
\caption{Reconstruction with $T=5$ and with added 50\% noise (in $L^{2}$
norm). In (e): gray line represents the phantom, dashed line shows image
reconstructed from the partial data, black line shows full data
reconstruction}
\label{F:t5partial}
\end{figure}

One of the advantages of the present method is the ease of implementation
using standard finite differences. Unlike algorithms of \cite{Stefanov-Yang}
, our approach does not require solving the Dirichlet problem for Laplace
equation to initialize the time reversal. (The latter problem is well
studied and various methods for its solution are known. However, efficient
numerical schemes for arbitrary domains are quite sophisticated and require
noticeable effort to implement).

Our numerical realization of the algorithm is based on equation (\ref
{E:final-alg}) that requires computing operators $\Lambda_{T}$ and $A$.
These operators represent solutions of the wave equation forward and
backwards in time, respectively; they were calculated using finite
difference stencils as described below.

Our simulations were performed on a 2D square domain $[-1,1]\times
\lbrack-1,1].$ Throughout this section we will represent our 2D spatial
variable $x$ in the coordinate form, and will change the notations for all
functions correspondingly:
\begin{equation*}
x=(\mathtt{x,y}),\quad u(x,t)=u(\mathtt{x,y},t),\quad v(x,t)=v(\mathtt{x,y}
,t),\quad\text{etc.}
\end{equation*}
Our square domain was discretized using Cartesian grid of size $257\times257$
, with the step $\Delta\mathtt{x}=\Delta\mathtt{y}=2/257;$ time was
discretized uniformly with the step $\Delta t=0.5\Delta\mathtt{x}$. The
speed of sound $c(\mathtt{x,y})$ was set to 1, for simplicity. Time stepping
inside the domain in the forward direction was implemented by applying
standard second-order centered stencils in both time and space to the
discretized solution $u(\mathtt{x}_{k},\mathtt{y}_{l},t_{j})$
\begin{align*}
\frac{\partial^{2}}{\partial\mathtt{x}^{2}}u(\mathtt{x}_{k},\mathtt{y}
_{l},t_{j}) & \thickapprox\widetilde{\frac{\partial^{2}}{\partial \mathtt{x}
^{2}}u(\mathtt{x}_{k},\mathtt{y}_{l},t_{j})}\equiv\frac {u(\mathtt{x}_{k+1},
\mathtt{y}_{l},t_{j})+u(\mathtt{x}_{k-1},\mathtt{y}_{l},t_{j})-2u(\mathtt{x}
_{k},\mathtt{y}_{l},t_{j})}{\Delta\mathtt{x}^{2}}, \\
\frac{\partial^{2}}{\partial\mathtt{y}^{2}}u(\mathtt{x}_{k},\mathtt{y}
_{l},t_{j}) & \thickapprox\widetilde{\frac{\partial^{2}}{\partial \mathtt{y}
^{2}}u(\mathtt{x}_{k},\mathtt{y}_{l},t_{j})}\equiv\frac {u(\mathtt{x}_{k},
\mathtt{y}_{l+1},t_{j})+u(\mathtt{x}_{k},\mathtt{y}_{l-1},t_{j})-2u(\mathtt{x
}_{k},\mathtt{y}_{l},t_{j})}{\Delta\mathtt{x}^{2}}, \\
\frac{\partial^{2}}{\partial t^{2}}u(\mathtt{x}_{k},\mathtt{y}_{l},t_{j}) &
\thickapprox\widetilde{\frac{\partial^{2}}{\partial t^{2}}u(\mathtt{x}_{k},
\mathtt{y}_{l},t_{j})}\equiv\frac{u(\mathtt{x}_{k},\mathtt{y}_{l},t_{j+1})+u(
\mathtt{x}_{k},\mathtt{y}_{l},t_{j-1})-2u(\mathtt{x}_{k},\mathtt{y}
_{l},t_{j})}{\Delta t^{2}},
\end{align*}
resulting in the formula
\begin{equation*}
u(\mathtt{x}_{k},\mathtt{y}_{l},t_{j+1})=2u(\mathtt{x}_{k},\mathtt{y}
_{l},t_{j})-u(\mathtt{x}_{k},\mathtt{y}_{l},t_{j-1})+\Delta t^{2}\left(
\widetilde{\frac{\partial^{2}}{\partial\mathtt{x}^{2}}u(\mathtt{x}_{k},
\mathtt{y}_{l},t_{j})}+\widetilde{\frac{\partial^{2}}{\partial \mathtt{y}^{2}
}u(\mathtt{x}_{k},\mathtt{y}_{l},t_{j})}\right) ,
\end{equation*}
where tilde denotes approximate quantities. Time stepping backwards in time
(when computing $A)$ were done similarly
\begin{equation}
v(\mathtt{x}_{k},\mathtt{y}_{l},t_{j-1})=2v(\mathtt{x}_{k},\mathtt{y}
_{l},t_{j})-v(\mathtt{x}_{k},\mathtt{y}_{l},t_{j+1})+\Delta t^{2}\left(
\widetilde{\frac{\partial^{2}}{\partial\mathtt{x}^{2}}v(\mathtt{x}_{k},
\mathtt{y}_{l},t_{j})}+\widetilde{\frac{\partial^{2}}{\partial \mathtt{y}^{2}
}v(\mathtt{x}_{k},\mathtt{y}_{l},t_{j})}\right) .   \label{E:inside}
\end{equation}

When computing action of the operator $\Lambda_{T}$ (forward problem), the
Neumann boundary condition was represented by the simplest first-order
two-point stencil; this results in the values at the boundary points being
set to the values of the nearest grid points.

The discretization of the non-standard boundary condition
\begin{equation}
\frac{\partial}{\partial\nu}v(x,t)-\lambda(x)\,\frac{\partial}{\partial t}
v(x,t)=-\lambda(x)\,\frac{\partial}{\partial t}g(x,t),
\label{E:smartboundary}
\end{equation}
arising in problem (\ref{E:Rev-p}) was performed as follows. The simplest
first-order two-point forward stencils were used to approximate all the
derivatives; for example $\frac{\partial}{\partial t}v$ was approximated at $
t=t_{j-1}$, \ $\mathtt{x}=\mathtt{x}_{0}$\ by
\begin{equation*}
\frac{\partial}{\partial t}v(\mathtt{x}_{0},\mathtt{y}_{l},t_{j-1})
\thickapprox\widetilde{\frac{\partial}{\partial t}v(\mathtt{x}_{0},\mathtt{y}
_{l},t_{j-1})}\equiv\frac{v(\mathtt{x}_{0},\mathtt{y}_{l},t_{j})-v(\mathtt{x}
_{0},\mathtt{y}_{l},t_{j-1})}{\Delta t},
\end{equation*}
and $\frac{\partial}{\partial t}g$ was computed similarly. The normal
derivative was also approximated by the simplest two-point stencil applied
to values at time $t_{j-1}$ ; for example, one the side with $\mathtt{x}=
\mathtt{x}_{0}$ the following formula was used:
\begin{equation*}
\frac{\partial}{\partial\nu}v(\mathtt{x}_{0},\mathtt{y}_{l},t_{j-1})=-\frac{
\partial}{\partial\mathtt{x}}v(\mathtt{x}_{0},\mathtt{y}_{l},t_{j-1})
\thickapprox-\widetilde{\frac{\partial}{\partial\mathtt{x}}v(\mathtt{x}_{0},
\mathtt{y}_{l},t_{j-1})}\equiv-\frac{v(\mathtt{x}_{1},\mathtt{y}
_{l},t_{j-1})-v(\mathtt{x}_{0},\mathtt{y}_{l},t_{j-1})}{\Delta\mathtt{x}}.
\end{equation*}
Substituting the last two equations into (\ref{E:smartboundary}) resulted in
\begin{equation*}
\frac{v(\mathtt{x}_{0},\mathtt{y}_{l},t_{j-1})-v(\mathtt{x}_{1},\mathtt{y}
_{l},t_{j-1})}{\Delta\mathtt{x}}=\frac{\lambda(\mathtt{x}_{0},\mathtt{y}_{l})
}{\Delta t}\left[ v(\mathtt{x}_{0},\mathtt{y}_{l},t_{j})-v(\mathtt{x}_{0},
\mathtt{y}_{l},t_{j-1})-g(\mathtt{x}_{0},\mathtt{y}_{l},t_{j})+g(\mathtt{x}
_{0},\mathtt{y}_{l},t_{j-1})\right] \,
\end{equation*}
or
\begin{equation}
v(\mathtt{x}_{0},\mathtt{y}_{l},t_{j-1}) = v(\mathtt{x}_{1},\mathtt{y}
_{l},t_{j-1})+\gamma(\mathtt{x}_{0},\mathtt{y}_{l})\left[ v(\mathtt{x}_{0},
\mathtt{y}_{l},t_{j})-v(\mathtt{x}_{0},\mathtt{y}_{l},t_{j-1})-g(\mathtt{x}
_{0},\mathtt{y}_{l},t_{j})+g(\mathtt{x}_{0},\mathtt{y}_{l},t_{j-1})\right] ,
\label{E:smart1}
\end{equation}
where
\[
\gamma(\mathtt{x}_{0},\mathtt{y}_{l})  \equiv\frac{\lambda(
\mathtt{x}_{0},\mathtt{y}_{l})\Delta\mathtt{x}}{\Delta t}.  \notag
\]
Solving (\ref{E:smart1}) for $v(\mathtt{x}_{0},\mathtt{y}_{l},t_{j-1})$
yielded
\begin{equation}
v(\mathtt{x}_{0},\mathtt{y}_{l},t_{j-1})=\frac{v(\mathtt{x}_{1},\mathtt{y}
_{l},t_{j-1})}{1+\gamma(\mathtt{x}_{0},\mathtt{y}_{l})}+\frac{\gamma(\mathtt{
x}_{0},\mathtt{y}_{l}) \left[v(\mathtt{x}_{0},\mathtt{y}_{l},t_{j})-g(
\mathtt{x}_{0},\mathtt{y}_{l},t_{j})+g(\mathtt{x}_{0},\mathtt{y}_{l},t_{j-1})
\right]}{1+\gamma(\mathtt{x}_{0},\mathtt{y}_{l})}.   \label{E:smart2}
\end{equation}
Approximation of boundary condition (\ref{E:smartboundary}) on other parts
of the boundary was done similarly. In order to apply (\ref{E:smart2}) (and
similar expression on the other parts of the boundary), one first applies
(\ref{E:inside}) at all discretization points inside of the computational
domain. Then (\ref{E:smart2}) is fully defined.

In the absence of experimentally measured data, in order to validate the
reconstruction algorithm one needs to simulate values of $g(x,t)$ on $\Gamma.
$ One could use the finite difference algorithm described above to
approximately compute $g(x,t)$ for a chosen phantom $\mathbf{f}$ $=(f,0)$.
However, doing so would constitute the so-called \textquotedblleft inverse
crime\textquotedblright: sometimes simulations will produce inordinately
good reconstructions due to the spurious cancellation of errors if the
forward and direct problems are solved using the same discretization
techniques. Thus, in order to compute $g$ we utilized the following method
based on separation of variables. Function $f$ was expanded in the
orthogonal series of eigenfunctions $\varphi_{k,l}$ of the Neumann Laplacian
on our square domain
\begin{align*}
f(\mathtt{x,y}) & =\sum_{k,l}c_{k,l}\varphi_{k,l}(\mathtt{x,y}), \\
\varphi_{k,l}\mathtt{(x,y}) & =\cos(k\mathtt{\bar{x}})\cos(l\mathtt{\bar{y}}
),\quad k=0,1,2,...,\quad l=0,1,2,..., \\
\mathtt{\bar{x}} & =\pi(\mathtt{x}+1)/2,\quad\mathtt{\bar{y}}=\pi (\mathtt{y}
+1)/2.
\end{align*}
This was done efficiently using the 2D Fast Cosine Fourier transform
algorithm (FCT). Then, solution of the forward problem was computed as the
series
\begin{equation*}
u(\mathtt{x,y},t)=\sum_{k,l}c_{k,l}\varphi_{k,l}(\mathtt{x,y})\cos
(\lambda_{k,l}t),\quad\lambda_{k,l}=\frac{\pi}{2}\sqrt{k^{2}+l^{2}},\quad
k,l=0,1,2,....
\end{equation*}
For each value of $t,$ the above 2D cosine series were also summed using the
FCT. The resulting algorithm is quite fast, and, more importantly, it is
spectrally accurate with respect to $f.$ If $f$ has high degree of
smoothness, this series solution yields much higher accuracy than the finite
difference techniques we utilized as parts of the reconstruction algorithm.

\subsection{Simulations}

\begin{figure}[t]
\begin{center}
\begin{tabular}{cc}
\includegraphics[width=2.0in,height=2.0in]{pics/t3_phantom.pdf} &
\includegraphics[width=2.0in,height=2.0in]{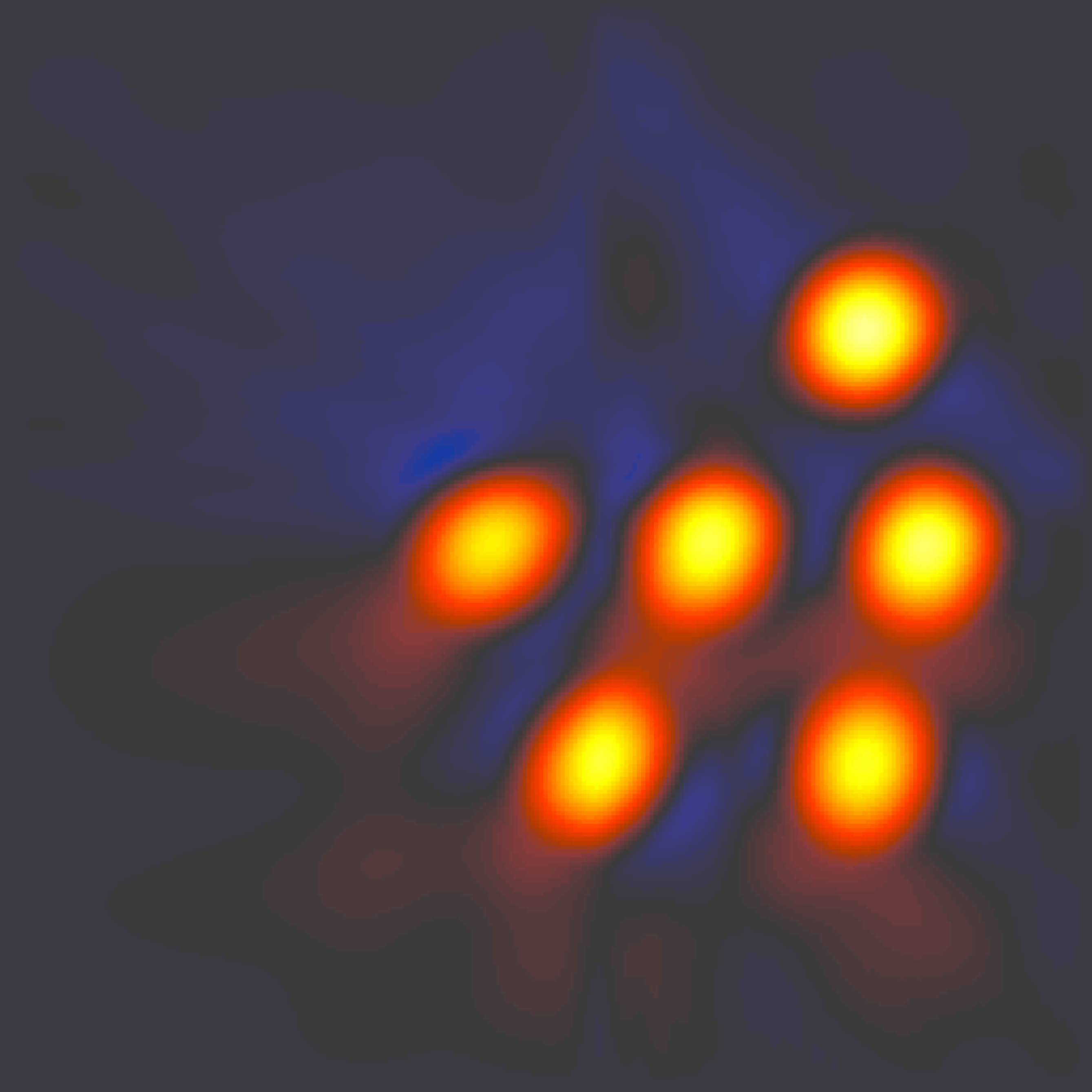} \\
(a) Phantom & (b) Initial approximation $\Pi_{1}Ag$ \\
&  \\
\includegraphics[width=2.0in,height=2.0in]{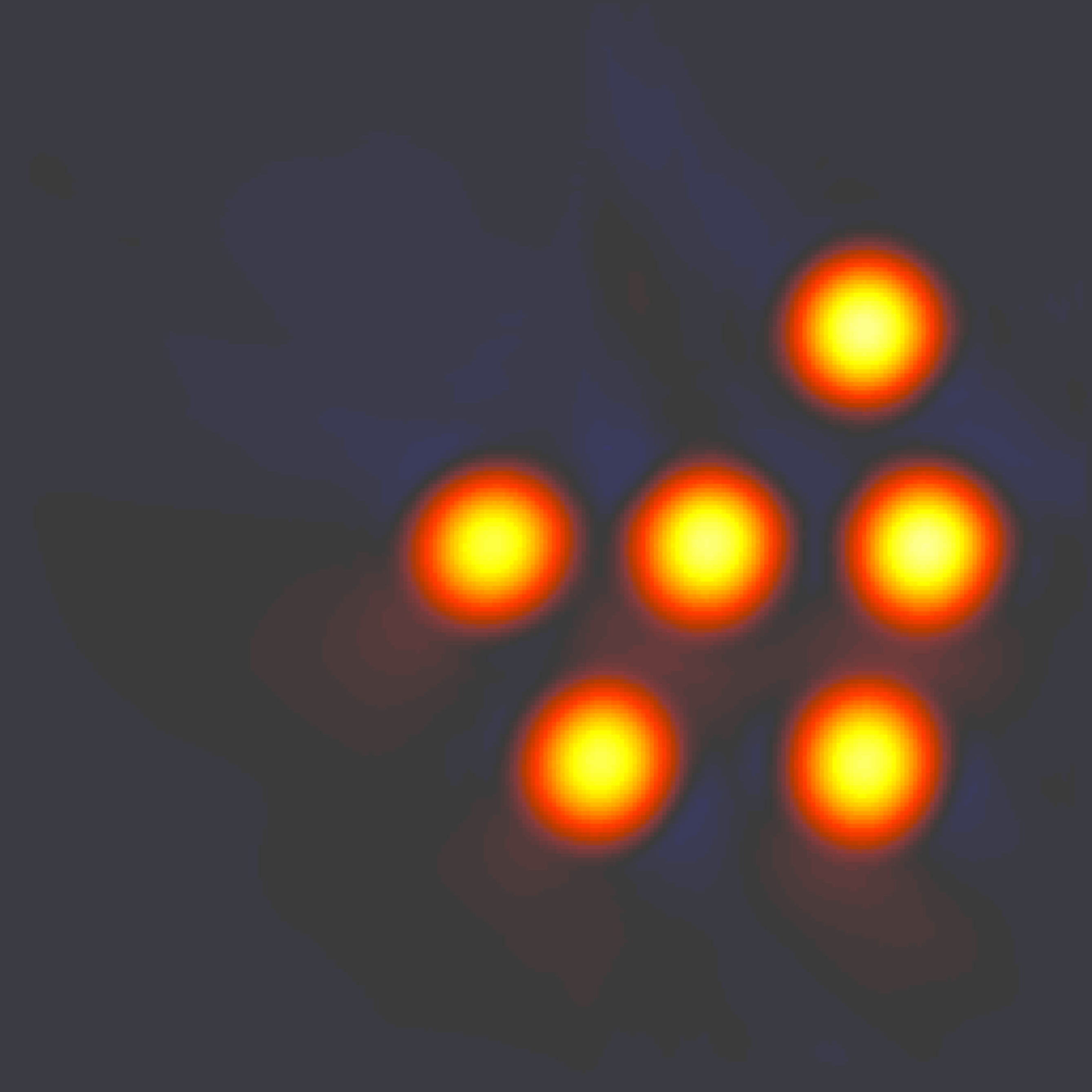} &
\includegraphics[width=2.0in,height=2.0in]{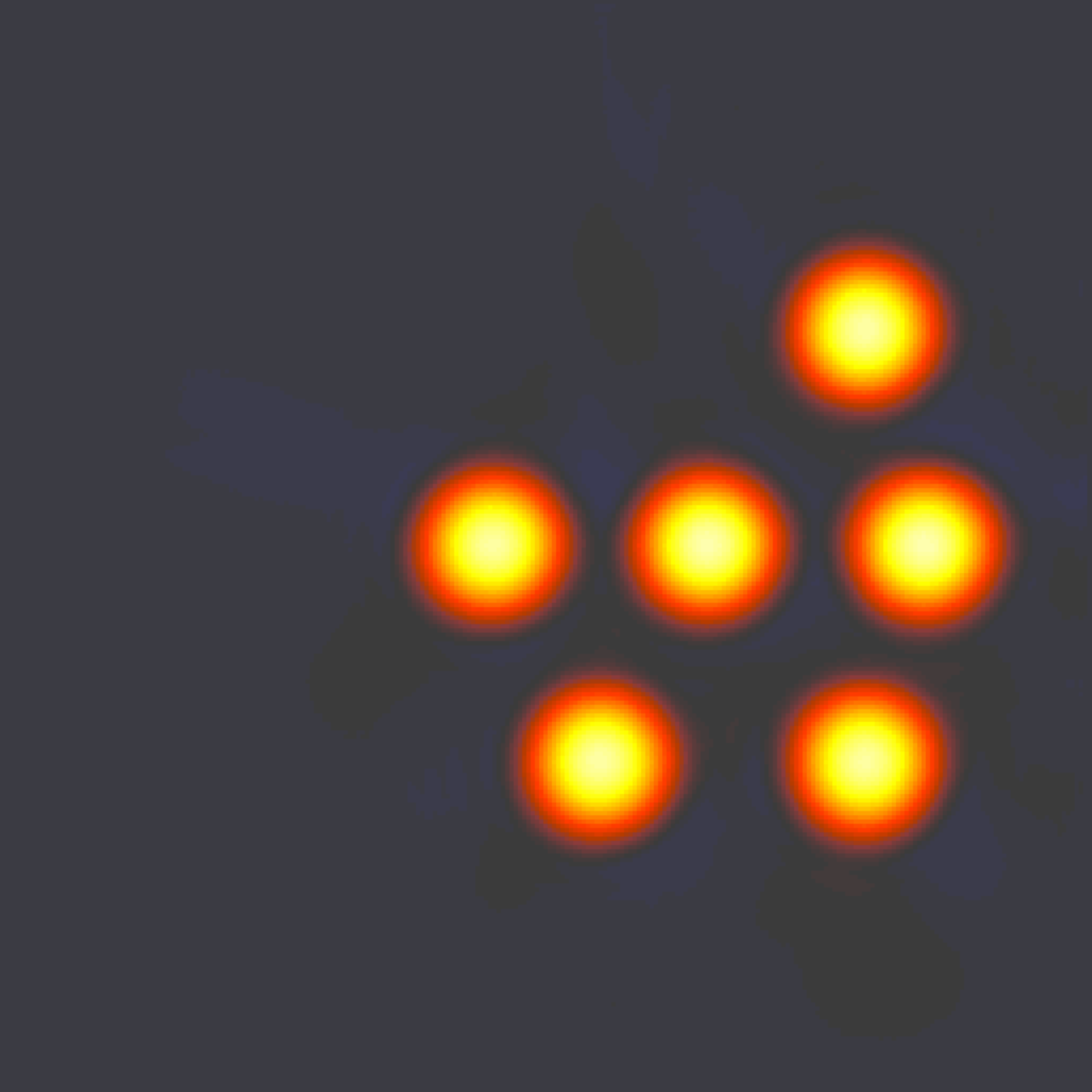} \\
(c) Iteration \#2, $\mathbf{u}^{(2)}$ & (d) Iteration \#5, $\mathbf{u}^{(5)}$
\\
&  \\
&
\end{tabular}
\\[0pt]
\includegraphics[width=4.6in,height=1.5in]{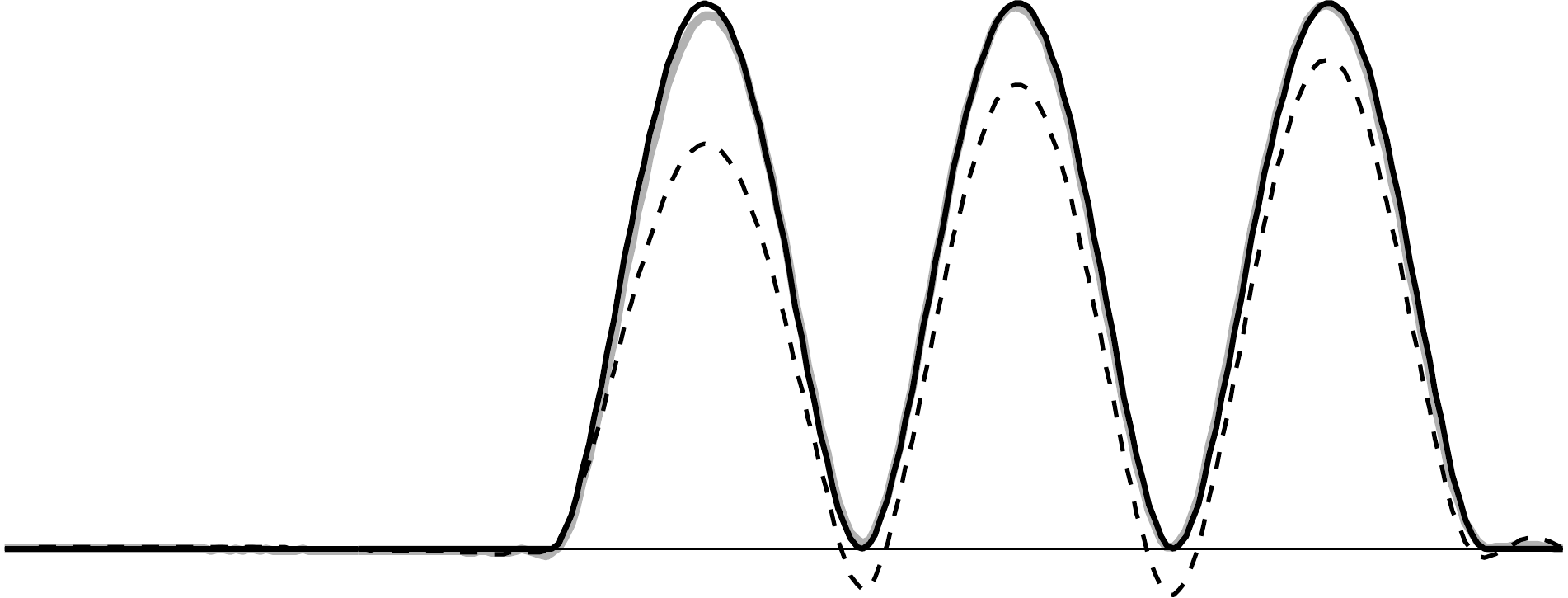}\\[0pt]
(e) Central horizontal cross-sections of (a), (b), and (d)
\end{center}
\caption{Reconstruction from the full boundary data, $T=1.6$. In (e): gray
line is the phantom, dashed line shows the initial approximation, black line
presents iteration \#5}
\label{F:fulldata}
\end{figure}

\begin{figure}[t]
\begin{center}
\begin{tabular}{cc}
\includegraphics[width=2.0in,height=2.0in]{pics/t3_phantom.pdf} &
\includegraphics[width=2.0in,height=2.0in]{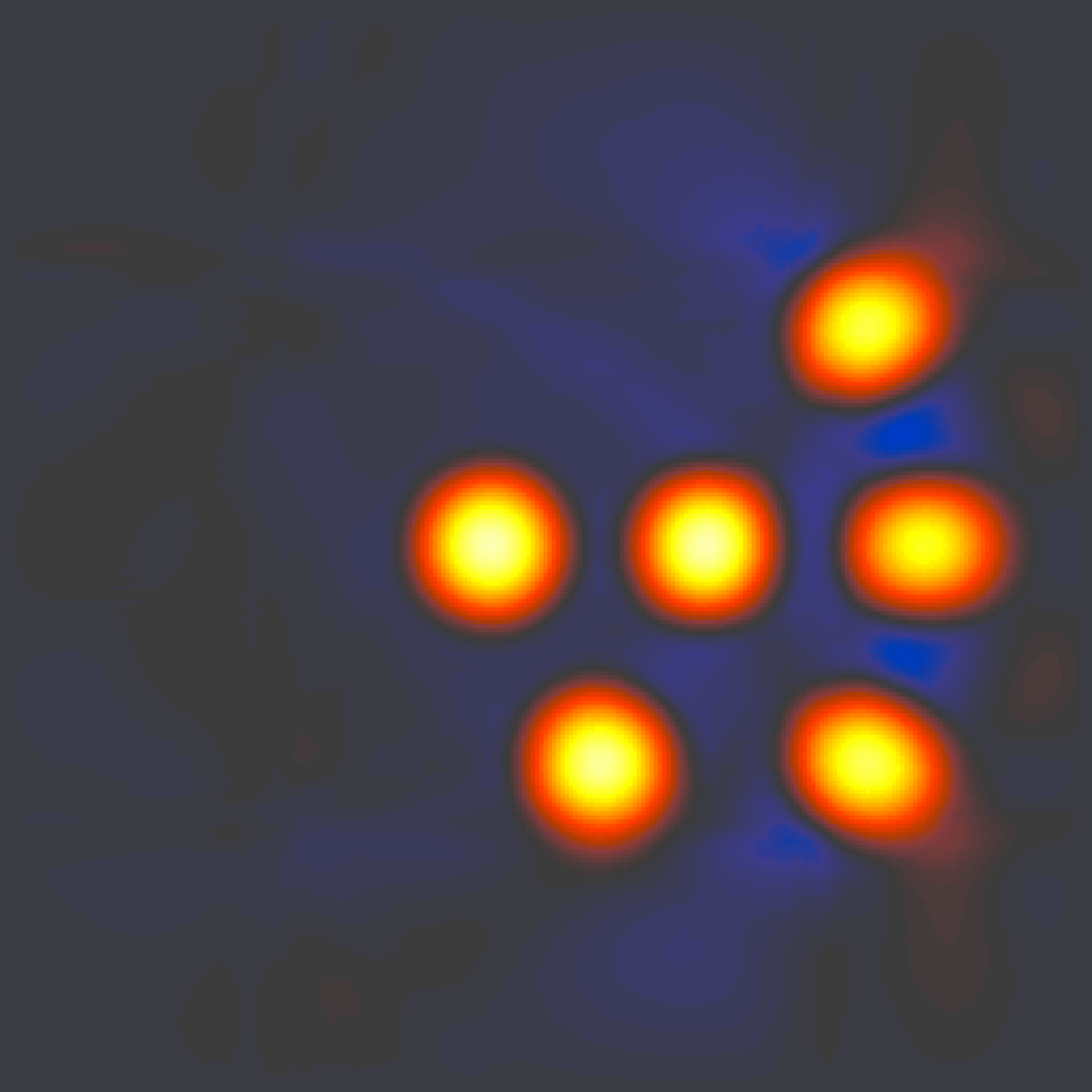} \\
(a) Phantom & (b) Initial approximation $\Pi_{1}Ag$ \\
&  \\
\includegraphics[width=2.0in,height=2.0in]{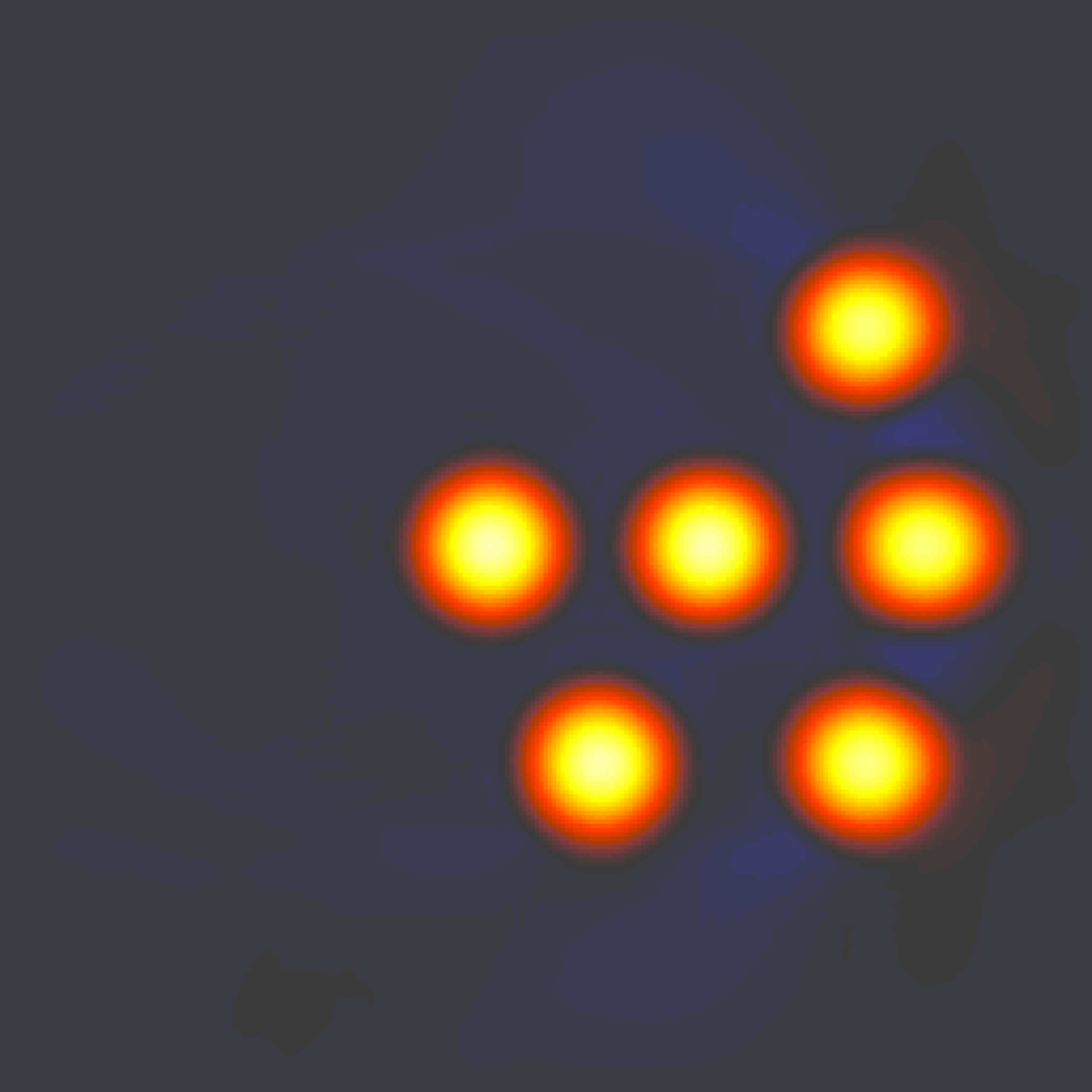} &
\includegraphics[width=2.0in,height=2.0in]{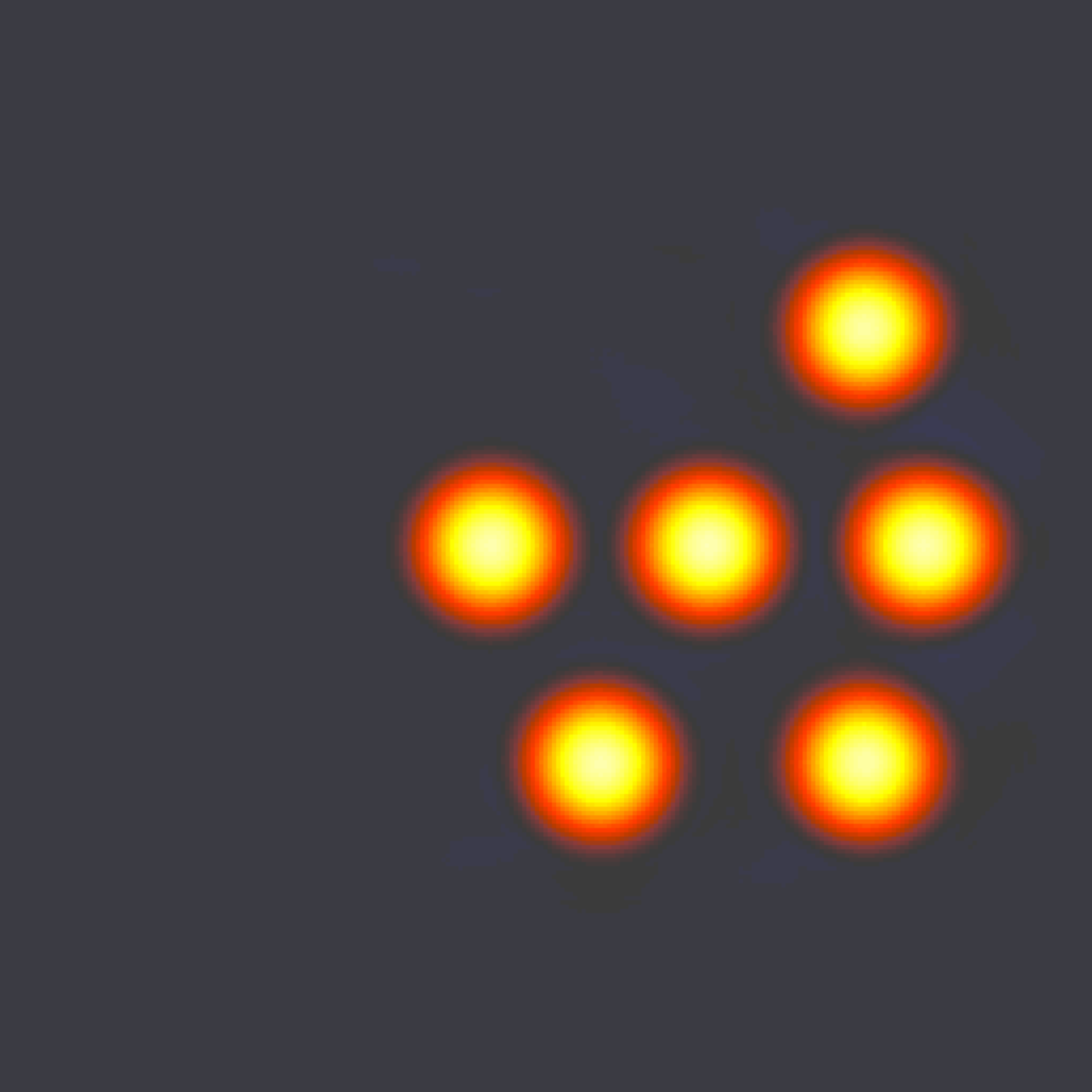} \\
(c) Iteration \#2, $\mathbf{u}^{(2)}$ & (d) Iteration \#5, $\mathbf{u}^{(5)}$
\\
&  \\
&
\end{tabular}
\\[0pt]
\includegraphics[width=4.6in,height=1.5in]{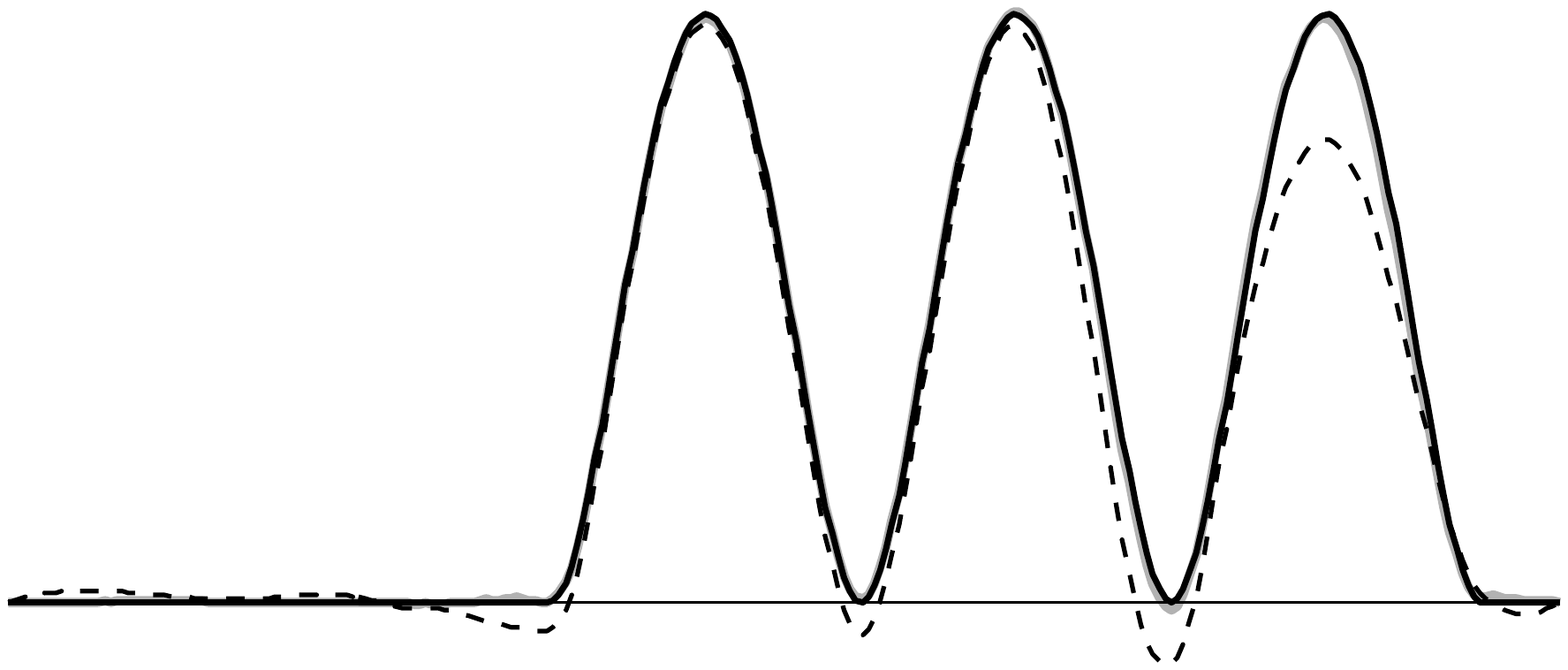}\\[0pt]
(e) Central horizontal cross-sections of (a), (b), and (d)
\end{center}
\caption{Reconstruction from the data given at the left and bottom sides of
the square domain, $T=3$. In (e): gray line is the phantom, dashed line
represents the initial approximation, black line shows iteration \#5}
\label{F:partialdata}
\end{figure}

We conducted several numerical experiments to verify theoretical conclusions
of previous sections. Two acquisition schemes were considered: a full data
scheme with $g(x,t)$ given on all sides of the square domain, and a partial
data scheme with $g$ known only on the left and bottom sides of the square.
Since we assumed $c(x)\equiv1,$ $T(\Omega,\Gamma)$ equals $2\sqrt{2}$ in the
former case (i.e., the length of the diagonal of the square) and $4\sqrt{2}$
in the latter case. Experimentally, we found that these times are too
pessimistic, and a half of that time is quite enough for the reconstruction.
This implies that our theoretical results may be not sharp. Although we were
not able to improve these theoretical estimates, we present below
simulations with the measurement times significantly reduced compared to $
T(\Omega,\Gamma)$.

As a phantom, we utilized a sum of six shifted finitely supported $
C^{1}(\Omega)$ functions of the radial variable, shown as a color-scale
image in Figure~1(a). Such a smooth phantom reduces errors related to finite
difference computations and allows us to concentrate on convergence of the
algorithm per se.

The goal of our first two simulations was to see how well the initial
approximation $\Pi_{1}Ag$ can approximate $f.$ To this end the measurement
time $T$ was chosen to equal to 5; this corresponds to the time of two and a
half bounces of waves between the opposite sides of the domain. \
Reconstruction from full data is shown in Figure~1(b) and partial data
reconstruction is presented in Figure~1(c). Image in Figure~1(d)
demonstrates profiles of the central horizontal cross sections of the three
previous images, with the phantom represented by a gray line, full time
reconstruction shown as a black line, and partial data image drawn with a
dashed line. The full data image is practically perfect; the corresponding
black line in Figure~1(d) lying on top of the gray line, rendering the
latter almost invisible. The relative $L^{2}(\Omega)$ reconstruction error
equals 1.1\% \ in the image of Figure~1(b). The partial data reconstruction
is just slightly less accurate, with the relative $L^{2}(\Omega)$ error
equal to 6.8\%. In any reconstruction from real data such error would be
negligible compared to errors introduced by imperfections of real data.

In order to illustrate the low sensitivity of the algorithm to the noise in
the data, we repeated the previous simulation with the data contaminated by
50\% white noise (in the relative $L^{2}$ norm). As in the first simulation,
only the initial guess $\Pi_{1}Ag$ was computed, without the successive
iterative refinement. Figure~2(a) shows the time series representing $g(x,t)$
for one of the points $x,$ with and without added noise. Figure~2(b)-(d)
demonstrate the same phantom, and the full and partial data reconstructions,
respectively. The errors in the two reconstructions were of about the same
order, 19\% and 22\% in the relative $L^{2}(\Omega)$ norm, with the partial
data giving, for some reason, slightly better result in this norm. The
central cross sections of these images are shown in Figure~2(e).

Our remaining two simulations were intended to verify the convergence of the
algorithm in the case when measurement time $T$ is close to a half of $
T(\Omega,\Gamma)$.  Figures~3(a)-(e) demonstrate results of a full data
reconstruction with $T$ equal to 1.6 (compare to $T(\Omega,\Gamma)=2\sqrt{2}
\thickapprox 2.828).$ Figures~3(a)-(d) show the phantom, the first
approximation $\Pi _{1}Ag$, and the second and the fifth iterations
$(\mathbf{u}^{(2)}$ and $\mathbf{u}^{(5)})$, respectively. Figure~3(e)
presents central horizontal cross-sections of the phantom, the first
approximation $\Pi_{1}Ag$, and of the fifth iterations. One can notice that,
while the initial approximation had been noticeably distorted, the fifth
iteration yields a close approximation to $f(x).$ The relative $L^{2}(\Omega)
$ norm of the error in $\mathbf{u}^{(5)}$ was $3.3\%.$

The final series of images in Figure~4 demonstrates the results of the
reconstruction from the partial data, with $T$ equal to $3$ (compare to $
T(\Omega,\Gamma)=4\sqrt{2}\thickapprox5.6569).$ As before, the phantom, the
initial approximation, and the second and the fifth iterations are shown in
Figure~4(a)-(d), respectively. Figure~4(e) presents the central horizontal
cross sections of images in Figure~4(a), (b), and (d). The relative $
L^{2}(\Omega)$ error in the fifth iteration $\mathbf{u}^{(5)}$ was $5.4\%;$
for most practical purposes this would be more than acceptable.

\section{Conclusions}

We presented a novel dissipative time reversal approach for solving the
inverse source problem of TAT/PAT posed within a cavity with perfectly
reflecting walls. Unlike the previous work of \cite
{Holman-Kunyansky,Stefanov-Yang} where Dirichlet boundary condition was
used, we utilize the non-standard boundary condition (\ref{E:Rev-p}). The
latter leads to the dissipative boundary condition (\ref{E:err}) imposed on
the error $U(x,t)$, and, hence, to a natural decay of $U(x,0)$ with the
growth of $T$. Our approach results in two reconstruction methods: i) a
non-iterative approximation, converging exponentially to $f$ as $T \to \infty
$, and ii) a Neumann series formula. These two algorithms are applicable for
both full and partial data problems.

Compared to \cite{Holman-Kunyansky}, where rather stringent conditions on
the eigenvalues of the Neumann and Dirichlet Laplacians on $\Omega$ are
required for convergence, our approach is based on the much less restrictive
GCC (Condition~\textbf{\ref{A:Gcc}}). Moreover, unlike the method of~\cite
{Stefanov-Yang}, our technique does not require computing the harmonic
extension of the boundary values, which significantly simplifies its
numerical realization. It should be noted that the requirement $T \geq T(\Omega,\Gamma)$ is sharp for the convergence in Theorem~\ref{T:Neumann-ser} (dealing with the general Problem~\ref{P:TATg}). However, it is not sharp for the convergence in Theorem~\ref{T:TAT} (dealing with the inverse problem of TAT/PAT). Specifically, for the case of the full data, it is twice of the sharp time for the convergence of the TAT/PAT's inverse problem, obtained in~\cite{Stefanov-Yang}.  Nevertheless, our numerical simulations show that the present algorithm performs very well with such sharp measurement times.

While we only considered the simplest wave equation in the Euclidean spaces,
our analysis can be extended to problems formulated on Riemannian manifolds
(as in \cite{Stefanov-Yang}) and/or to problems with a potential (as in \cite
{Acosta}).

\section*{Acknowledgment}

The first author is grateful to C. Bardos, G. Nakamura, and P. Stefanov for
helpful discussions and comments. The second author thanks L. Friedlander
for a helpful discussion. The first and second authors were partially
supported by the NSF/DMS awards \# 1212125 and 1211521, respectively.

\end{document}